\pdfoutput=1
\documentclass[reqno]{amsart}
\usepackage{amsmath,calligra,mathrsfs}
\usepackage{amsthm}
\makeatletter
\def\els@aparagraph[#1]#2{\elsparagraph[#1]{#2\@addpunct{.}}}
\def\els@bparagraph#1{\elsparagraph*{#1\@addpunct{.}}}
\makeatother
\usepackage{amssymb}
\usepackage[left=1.25in, right=1.25in]{geometry}
\usepackage{tikz}[2010/10/13]
\usetikzlibrary{decorations.pathreplacing}
\usepackage[all]{xy}
\usepackage{graphicx}
\usepackage{indentfirst}
\usepackage{mathtools}
\usepackage{bm}
\usepackage{bbold}
\usepackage{bbm}
\usepackage{mathrsfs}
\usepackage{latexsym}
\usepackage{hyperref}
\usepackage{microtype}
\usepackage{float}
\usetikzlibrary{shapes.geometric}
\usepackage{ifthen}
\usepackage{enumerate}

\DeclareMathOperator{\GHZ}{GHZ}

\begin{document}
		\title{A generating problem for subfactors}
	\author{Yunxiang Ren} 
	\address{Department of Mathematics and Department of Physics, Harvard University}
	\email{\href{yren@g.harvard.edu}{yren@g.harvard.edu}}
	\maketitle
	
	\newtheorem{Lemma}{Lemma}
	\theoremstyle{plain}
	\newtheorem{theorem}{Theorem~}[section]
	\newtheorem*{main}{Main Theorem}
	\newtheorem*{utheorem}{Theorem}
	\newtheorem{lemma}[theorem]{Lemma~}
	\newtheorem{question}[theorem]{Question~}
	\newtheorem{proposition}[theorem]{Proposition~}
	\newtheorem{corollary}[theorem]{Corollary~}
	\newtheorem{definition}[theorem]{Definition~}
	\newtheorem{notation}[theorem]{Notation~}
	\newtheorem{example}[theorem]{Example~}
	\newtheorem*{remark}{Remark}
	\newtheorem{cor}[theorem]{Corollary~}
	\newtheorem*{claim}{Claim}
	\newtheorem*{conjecture}{Conjecture~}
	\newtheorem*{fact}{Fact~}
	\renewcommand{\proofname}{\bf Proof}
	\newcommand{\FS}{\mathfrak{F}}
	\newcommand{\bZ}{\mathbb{Z}}
	\newcommand{\bR}{\mathbb{R}}
	\newcommand{\bN}{\mathbb{N}}
	\newcommand{\bC}{\mathbb{C}}
	\newcommand{\fA}{\mathfrak{A}}
	\newcommand{\fB}{\mathfrak{B}}
	\newcommand{\fr}{r_\mathfrak{A}}
	\newcommand{\cN}{\mathcal{N}}
	\newcommand{\cM}{\mathcal{M}}
	\newcommand{\cH }{\mathcal{H}}
	\newcommand{\cC}{\mathcal{C}}
	\newcommand{\cR}{\mathcal{R}}
	\newcommand{\cS}{\mathcal{S}}
	\newcommand{\cA}{\mathcal{A}}
	\newcommand{\cB}{\mathcal{B}}
	\newcommand{\sC}{\mathscr{C}}
	\newcommand{\sP}{\mathscr{P}}
	\newcommand{\fF}{\mathfrak{F}}
	\newcommand{\kG}[1]{\mathscr{P}^{KG_{#1,2}}}
	\newcommand{\norm}[1]{\| #1\|}
	\renewcommand{\geq}{\geqslant}
	\renewcommand{\leq}{\leqslant}

\begin{abstract}
	Bisch and Jones proposed the classification of planar algebras by simple generators and relations. In this paper, we study the generating problem for a family of group-subgroup subfactors associated with the Kneser graphs, namely, to determine the generators with minimal size. In particular, we prove that this family of subfactors are generated by $2$-boxes and this provides an affirmative answer to a question of Vaughan Jones. This generator problem is also related to the theory of quantum permutation groups, and the main theorem also provides an infinite family of strongly regular graphs with no quantum symmetry.    
\end{abstract}

\section{Introduction}\label{sec: intro}
Vaughan Jones initialled modern subfactor theory by his remark index theorem \cite{Jon83}. Since then, there are many different formalisms to understand the central object, namely, the \textit{standard invariants} for subfactors \cite{Bis97, Pop95, Ocn88}. Later on, Vaughan Jones introduced the subfactor planar algebras as a topological axiomatization of standard invariants \cite{Jon99}. A planar algebra $\sP_\bullet$ consists of a sequence of finite-dimensional $C^*$-algebras $\sP_{m,\pm}$ (which are called the $m$-box spaces) and a natural action of the operad of planar tangles. This perspective displays that the \textit{standard invariants} is a representation of \textit{fully labeled planar tangles} in the flavor of \textit{topological quantum field theory} \cite{Ati88}. 

From the perspective of planar algebras, Bisch and Jones proposed the classification of subfactors by \textit{simple generators and relations} \cite{BisJon00, BisJon03, BJL17}. The motivating examples are the Birman-Murakami-Wenzl algebras \cite{BirWen89, Mur90}, which admits the Yang-Baxter relations. Such Yang-Baxter relation planar algebras are completely classified in \cite{LiuYB}, and a new family of subfactor planar algebras were discovered there which has a deep connection to conformal field theory. For Yang-Baxter relations, the critical dimension of the $3$-box space is 15: the dimension of the $3$-box space of singly-generated Yang-Baxter relation planar algebras is less than or equal to $15$. Moreover, when the dimension is less than 15, Yang-Baxter relations will always hold. However, when the dimension is 15, Yang-Baxter relations are not automatic anymore. Therefore, do there exists singly-generated planar algebras beyond Yang-Baxter relations? In particular, Jones asked whether the subfactor planar algebra for $S_2\times S_3\subset S_5$ is generated by its $2$-boxes in the late nineties. In \cite{RenThesis, Ren19Uni}, we provide an affirmative answer to this question and a skein theory from the perspective of group-action models. Later on, Jones asked the following question.
\begin{question}[Jones, 2017]\label{quest: VJ}
	Are the subfactor planar algebras for $S_2\times S_{n-2}< S_n$ generated by their $2$-boxes? 
\end{question} 
This question is also closely related to the classification of spin models for Kauffman polynomial from self-dual strongly regular graphs by Jaeger \cite{Jae92}. In particular, he discovered a new spin model based on the Higman-Sims graph \cite{HigSims68}. The spin model is described by a \textit{spin model planar algebra}, and the adjacency matrix is a $2$-box. Question \ref{quest: VJ} can be asked in this general setup: given a strongly regular graph $\Gamma$, the associated group-action model $\sP_\bullet^\Gamma$ is defined to be the fixed-point planar subalgebra of the spin model planar algebra. The $2$-box space is spanned by Temperley-Lieb diagrams and the adjacency matrix $A_\Gamma$. Therefore, one can ask whether the planar algebra $\sP_\bullet^\Gamma$ is generated by the adjacency matrix $A_\Gamma$. Since the spin model planar algebra $\sP^\Gamma_\bullet$ is defined by the combinatorial data of the graph $\Gamma$, the generating property in Question \ref{quest: VJ} is intrinsically determined by $\Gamma$. 
\begin{definition}\label{def: property G}
	Let $\Gamma$ be a strongly regular graph. We say $\Gamma$ has property $(G)$ if the associated planar algebra $\sP^\Gamma$ has the generating property, namely, it is generated by its adjacency matrix $A_\Gamma$. 
\end{definition}
The referred subfactor planar algebra in Question \ref{quest: VJ} can be obtained from the Kneser graph $KG_{n,2}$. Therefore, Question \ref{quest: VJ} is equivalent to ask whether $KG_{n,2}$ has property $(G)$ for $n\geq 5$. In \cite{RenThesis}, we provide an affirmative answer to Question \ref{quest: VJ} in the case when $n=5$, namely, the Petersen graph $KG_{5,2}$ has property $(G)$. 

In this paper, by exploiting the universal skein theory for group-action models \cite{Ren19Uni}, we first give constructions of generators for the planar algebras $\kG{n}_\bullet$ under the assumption that the transposition $R$ is generated by $2$-boxes, namely, $R\in \langle \kG{n}_2\rangle$. Then we confirm the validity of the assumption provided with a universal construction, and thus we prove the main theorem, namely, 
\begin{main}
	The Kneser graph $KG_{n,2}$ has property $(G)$ for $n\geq 5$. 
\end{main}
We first remark that the relation between the transposition and the generating property was first revealed independently by Jones \cite{JonPC} and Curtin \cite{Cur03}. They showed that any planar subalgebra $\mathscr{Q}_\bullet$ of some spin model, $\mathscr{Q}$ has the generating property if and only if $R\in\langle \mathscr{Q}_2\rangle$. We enhance the statement by dropping the assumption that $\mathscr{Q}$ is a planar subalgebra of some spin model. In this case, the transposition $R$ is characterized by skein relations. 

Secondly, it was pointed out by Snyder and Reutter that the generating property in Definition \ref{def: property G} is also studied in the theory of quantum permutation groups \cite{LMR17, Cha19, MRV19}. Quantum automorphism groups were defined by Banica (See e.g.\cite{Ban05}). A graph $\Gamma$ is said to have \textit{no quantum symmetry} if its quantum automorphism group coincides with its automorphism group. Moreover, we have that
\begin{equation}
\Gamma\text{ has property $(G)$} \Longleftrightarrow \Gamma \text{~has no quantum symmetry.}
\end{equation}
In the theory of quantum permutation groups, It is an important task to determine graphs having no quantum symmetry. In \cite{BanBic07}, Banica and Bichon computed the quantum automorphism groups for strongly regular graphs with vertices less than or equal to 11, except for the Petersen graph. Therefore, Main theorem confirms that the Petersen graph has no quantum symmetry and provide an infinite family of strongly regular graphs with no quantum symmetry, namely,
\begin{cor}
	The Kneser graphs $KG_{n,2}$ has no quantum symmetry for $n\geq 5$. 
\end{cor}
In the end, Main Theorem confirms that the simplest generator for the planar algebra for $\kG{n}_\bullet$ is a single $2$-box $A_\Gamma$. However, Universal skein theory for group actions tells us that in the simplest skein theory, the generators are a $2$-box and an $n$-box; and one of the relations appears in the $2n$-box space. This phenomenon gives us a hint that the complexity of skein theory might be more subtle than the sizes of generators and relations. 

\textbf{Acknowledgment.} The author would like to thank Vaughan Jones for support, encouragement, and many inspiring conversations and Arthur Jaffe, Zhengwei Liu for many helpful discussions. The author would also like to thank Noah Snyder and David Reutter for pointing out the connection to the theory of quantum permutation group. The research was partially supported by DMS-1362138 and also supported by TRT 0159 from the Templeton Religion Trust.
\section{preliminary}\label{sec: pre}
In this section, we recall the basics of spin models and group-action models. We refer the readers for more details to \cite{Jon99} for spin model planar algebras and \cite{Ren19Uni} for group-action models. 
\begin{definition}[Spin models]\label{def: spin model}
	Let $X$ be a finite set with size $d$. The spin model $\mathscr{P}_\bullet$ associated to $X$ is a family of vector spaces $\{\mathscr{P}_n:n\geq 0\}$, where $\mathscr{P}_n=\mathcal{F}(X^k,\mathbb{C})$, namely, the complexed-valued functions on $X^k$. Moreover, there are three basic operations on $\mathscr{P}_\bullet$:
	\begin{itemize}
		\item Tensor product$\colon$
		\begin{equation}\label{equ: tensor product}
		f\otimes g(x_1,x_2,\cdots,x_n,y_1,y_2,\cdots,y_m)=f(x_1,x_2,\cdots,x_n)g(y_1,y_2,\cdots,y_m).
		\end{equation}
		\item Contraction$\colon$ for $1\leq k\leq n$,
		\begin{equation}\label{equ: contraction}
		C_{k,k+1}(f)(x_1,x_2,\cdots,x_n)=\delta_{x_k,x_{k+1}}\cdot (x_1,x_2,\cdots,x_{k-1},x_{k+2},\cdots,x_n),
		\end{equation}
		where $\delta$ is the Kronecker delta.
		\item Permutation$\colon$ for $1\leq k\leq n$,
		\begin{equation}\label{equ: permutation}
		S_{k,n}(f)(x_1,x_2,\cdots,x_n)=(x_1,x_2,\cdots,x_{k-1},x_{k+1},x_k,x_{k+2},\cdots,x_n).
		\end{equation} 
	\end{itemize}    
\end{definition}
\begin{definition}[Group-action models]\label{def: group-action model}
	Let $X$ be a finite set of size $d$ and $\mathscr{P}_\bullet$ be the spin model associated to $X$. Suppose there exists an action $\beta$ of a finite group $G\leq S_d$ on the set $X$ and thus the action $\beta$ can be extended diagonally on $X^n$ by 
	\begin{equation}
	\beta(g)(x_1,x_2,\cdots,x_n)=(\beta(g)x_1,\beta(g)x_2,\cdots,\beta(g)x_n).
	\end{equation}
	Therefore, this induces an action of $G$ on the spin model $\mathscr{P}_\bullet$, still denoted by $\beta$. The group $G$ is indeed the gauge symmetry of the spin model. The group-action model, denoted by $\mathscr{P}^G_\bullet$ is defined to be fixed points under the group action $\beta$, namely,
	\begin{equation}
	\mathscr{P}^G_n=\{f\in\mathscr{P}_n: \beta(g)f=f,~\forall~g\in G\}.
	\end{equation}
	It is straightforward to verify that the group-action model is closed under the three operations defined in Definition \ref{def: spin model}. 
\end{definition}
In \cite{Ren19Uni}, we provide a universal skein theory for group-action models. Here, we recall the generators in the following proposition. 
\begin{proposition}\label{pro: diagrammatic representation}
	Let $\mathscr{P}_\bullet^G$ be a group-action model associated to the group action $G\curvearrowright X$ where $X$ is a set of size $d$. Suppose $f\in\mathscr{P}_n^G$ for some $n\in\mathbb{N}$. We represent $f$ as the following diagram
	
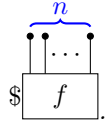
\begin{figure}[H]
		\begin{tikzpicture}
		\draw (0,0) rectangle (1,-.6);
		\node at (.5,-.3) {$f$};
		\node at (-.1,-.3) {$\$$};
		\draw (.1,0)--(.1,.5);
		\draw (.3,0)--(.3,.5);
		\draw (.9,0)--(.9,.5);
		\draw [fill=black] (.1,.5) circle [radius=.04];
		\draw [fill=black] (.3,.5) circle [radius=.04];
		\draw [fill=black] (.9,.5) circle [radius=.04];
		\node at (.6,.25) {$\cdots$};
		\draw[decoration={brace,raise=5pt},decorate, thick, blue]
		(.1,.45) -- node[above=6pt] {$n$} (.9,.45);
		\end{tikzpicture}.
	\end{figure}
	The group-action model $\mathscr{P}_\bullet^G$ is generated by 
	\begin{itemize}
		\item The $\GHZ$ tensor: let $I_3=\{(x,x,x): x\in X\}$. The $\GHZ$ tensor is defined to be $\chi_{I_3}$, namely, the characteristic function of $I_3$. Moreover, the $\GHZ$ tensor is represented by the following diagram.
		\begin{equation}
		\begin{tikzpicture}
		\draw (0,0)--(0,.5);
		\draw (-.5,-.5)--(0,0)--(.5,-.5);
		\draw [fill=black] (0,0) circle [radius=.04];
		\draw [fill=black] (0,.5) circle [radius=.04];
		\draw [fill=black] (-.5,-.5) circle [radius=.04];
		\draw [fill=black] (.5,-.5) circle [radius=.04];
		\node at (-.75,0) [left] {$\GHZ=$};
		\end{tikzpicture}.
		\end{equation}
		\item The transposition $R$: For an arbitrary point $(x_1,x_2,x_3,x_4)\in X^4$, we define the transposition $R$ by $R(x_1,x_2,x_3,x_4)=\delta_{x_1,x_3}\delta_{x_2,x_4}$. Moreover, $R$ is represented by the following diagram.
		\begin{equation}
		\begin{tikzpicture}
		\draw (-.5,-.5)--(.5,.5);
		\draw (-.5,.5)--(.5,-.5);
		\draw [fill=black] (-.5,-.5) circle [radius=.04];
		\draw [fill=black] (.5,-.5) circle [radius=.04];
		\draw [fill=black] (-.5,.5) circle [radius=.04];
		\draw [fill=black] (.5,.5) circle [radius=.04];
		\node at (-.75,0) [left] {$R=$};
		\end{tikzpicture}.
		\end{equation}
		\item The molecule $M$: let $S_M=\{(\beta(g)1,\beta(g)2,\cdots,\beta(g)d): g\in G\}$. The molecule $M$ is defined to be $\chi_{S_M}$, namely, the characteristic function of $S_M$. 
	\end{itemize}
\end{proposition}
\begin{remark}
	Let $H$ be the stabilizer of a single point $x\in X$. Then the group-action model $\sP^G_\bullet$ is the \textit{even part} of the group-subgroup subfactor planar algebra for $H\leq G$. In the language of subfactor planar algebras, the $\GHZ$ tensor is exactly the Temperley-Lieb diagram, \raisebox{-2mm}{
		\begin{tikzpicture}
		\begin{scope}[xscale=1.5, yscale=1.5]
		\path [fill=lightgray] (-.1,0) arc[radius=.1, start angle=180, end angle=0]--(.3,0)--(.3,.1) arc[radius=.1, start angle=0, end angle=90] arc[radius=.1, start angle=270, end angle=180]--(.1,.4)--(-.1,.4)--(-.1,.3) arc[radius=.1, start angle=0, end angle=-90] arc[radius=.1, start angle=90, end angle=180]--(-.3,0);
		\draw (-.1,0) arc[radius=.1, start angle=180, end angle=0];
		\draw (.3,0)--(.3,.1) arc[radius=.1, start angle=0, end angle=90] arc[radius=.1, start angle=270, end angle=180]--(.1,.4);
		\draw (-.1,.4)--(-.1,.3) arc[radius=.1, start angle=0, end angle=-90] arc[radius=.1, start angle=90, end angle=180]--(-.3,0);
		\end{scope}
		\end{tikzpicture}}. 
\end{remark}
\begin{remark}
	In general, the $\GHZ_n$ tensor is defined to be the characteristic function on the set
	\begin{equation}
	\{(x,x,\cdots,x):x\in X\}.
	\end{equation}
	Moreover, it is represented by the following diagram:
	
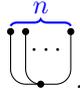
\begin{figure}[H]
		\begin{tikzpicture}
		\draw (.1,.5)--(.1,0) arc[radius=.2,start angle=180, end angle=270]--(.5,-.2) ;
		\draw (.3,.5)--(.3,0) arc[radius=.2,start angle=180, end angle=270];
		\draw (.5,-.2)--(.7,-.2) arc [radius=.2,start angle=270, end angle=360];
		\draw (.9,0)--(.9,.5);
		\draw [fill=black] (.5,-.2) circle [radius=.04];
		\draw [fill=black] (.1,.5) circle [radius=.04];
		\draw [fill=black] (.3,.5) circle [radius=.04];
		\draw [fill=black] (.9,.5) circle [radius=.04];
		\node at (.6,.25) {$\cdots$};
		\draw[decoration={brace,raise=5pt},decorate, thick, blue]
		(.1,.4) -- node[above=6pt] {$n$} (.9,.4);
		\end{tikzpicture}.
	\end{figure}
\end{remark}

\section{The generating property of subfactors}\label{sec: generating property}

The $2$-box space $\mathscr{P}_2$ is spanned by $\{I,J,A\}$, where $I$ is the identity matrix, $J$ is the matrix with each entry being 1 and $A$ is the adjacent matrix of $KG_{n,2}$. We represent these rank-$2$ tensors by the following diagrams:
\begin{figure}[H]
	\begin{tikzpicture}
	\draw (0,0)--(0,1);
	\draw [fill=black] (0,0) circle [radius=.04];
	\draw [fill=black] (0,1) circle [radius=.04];
	\draw [fill=black] (2,0) circle [radius=.04];
	\draw [fill=black] (2,1) circle [radius=.04];
	\draw (4,0)--(4,1);
	\draw [red] (4,.5) circle [radius=.04];
	\draw [fill=black] (4,0) circle [radius=.04];
	\draw [fill=black] (4,1) circle [radius=.04];
	\end{tikzpicture}.\caption{The rank-$2$ tensors $I$, $J$ and $A$.}
\end{figure}
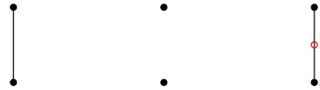
Moreover, the adjacent matrix $T$ for the complement of $KG_{n,2}$ is given by $T=J-I-A$ and we represent it as\raisebox{-4.2mm}{\begin{tikzpicture}
	\path [fill=white]  (-.5,0)rectangle (0,.2);
	\draw [cyan](0,0)--(0,1);
	\draw [fill=black] (0,0) circle [radius=.04];
	\draw [fill=black] (0,1) circle [radius=.04];
	\draw [ultra thick, cyan] (-.05,.5)--(.05,.5);
	\end{tikzpicture}}. Therefore the equation $J=I+T+A$ tells us that 
\begin{equation}
\begin{tikzpicture}
\begin{scope}
\draw (0,0)--(0,1);
\draw [fill=black] (0,0) circle [radius=.04];
\draw [fill=black] (0,1) circle [radius=.04];
\node at (.5,.5) {$+$};
\end{scope}
\begin{scope}[shift={(2,0)}]
\draw [cyan](0,0)--(0,1);
\draw [fill=black] (0,0) circle [radius=.04];
\draw [fill=black] (0,1) circle [radius=.04];
\draw [ultra thick, cyan] (-.05,.5)--(.05,.5);
\end{scope}
\begin{scope}[shift={(1,0)}]
\draw [red](0,0)--(0,1);
\draw [fill=black] (0,0) circle [radius=.04];
\draw [fill=black] (0,1) circle [radius=.04];
\draw [red](0,.5) circle [radius=.04]; 
\node at (.5,.5) {$+$};
\begin{scope}[shift={(-2,0)}]
\draw [fill=black] (0,0) circle [radius=.04];
\draw [fill=black] (0,1) circle [radius=.04];
\node at (.5,.5) {$=$};
\end{scope}
\end{scope}
\end{tikzpicture}.
\end{equation}
\begin{proposition}
	Let $R\in\mathscr{P}_4^G$ be the transposition. We define
	\begin{itemize}
		\item The element $R_A$ as follows:
		\begin{figure}[H]
			\begin{tikzpicture}
			\draw (-.5,-.5)--(.5,.5);
			\draw (-.5,.5)--(.5,-.5);
			\draw [fill=black] (-.5,-.5) circle [radius=.04];
			\draw [fill=black] (.5,-.5) circle [radius=.04];
			\draw [fill=black] (-.5,.5) circle [radius=.04];
			\draw [fill=black] (.5,.5) circle [radius=.04];
			\draw (-.25,.25)--(.25,.25);
			\draw [fill=black] (-.25,.25) circle [radius=.04];
			\draw [fill=black] (.25,.25) circle [radius=.04];
			\draw [red] (0,.25) circle [radius=.04];
			\node at (.75,0) [right] {$\in\mathscr{A}_4.$};
			\end{tikzpicture}
		\end{figure}
		\item The element $R_T$ as follows:
		
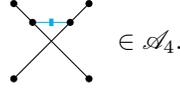
\begin{figure}[H]
			\begin{tikzpicture}
			\draw (-.5,-.5)--(.5,.5);
			\draw (-.5,.5)--(.5,-.5);
			\draw [fill=black] (-.5,-.5) circle [radius=.04];
			\draw [fill=black] (.5,-.5) circle [radius=.04];
			\draw [fill=black] (-.5,.5) circle [radius=.04];
			\draw [fill=black] (.5,.5) circle [radius=.04];
			\draw [cyan] (-.25,.25)--(.25,.25);
			\draw [fill=black] (-.25,.25) circle [radius=.04];
			\draw [fill=black] (.25,.25) circle [radius=.04];
			\draw [cyan,ultra thick] (0,.3)--(0,.2);
			\node at (.75,0) [right] {$\in\mathscr{A}_4.$};
			\end{tikzpicture}
		\end{figure}
	\end{itemize}
	Then we have that 
	\begin{equation}\label{equ: decompostion of R}
	\begin{tikzpicture}
	\begin{scope}[shift={(-4,0)}]
	\draw (-.5,-.5)--(.5,.5);
	\draw (-.5,.5)--(.5,-.5);
	\draw [fill=black] (-.5,-.5) circle [radius=.04];
	\draw [fill=black] (.5,-.5) circle [radius=.04];
	\draw [fill=black] (-.5,.5) circle [radius=.04];
	\draw [fill=black] (.5,.5) circle [radius=.04];
	\node at (1,0) {$=$};
	\end{scope}
	\begin{scope}[shift={(-2,0)}]
	\draw [fill=black] (-.5,-.5) circle [radius=.04];
	\draw [fill=black] (.5,-.5) circle [radius=.04];
	\draw [fill=black] (-.5,.5) circle [radius=.04];
	\draw [fill=black] (.5,.5) circle [radius=.04];
	\draw [fill=black] (0,0) circle [radius=.04];
	\draw (.5,.5)--(0,0)--(.5,-.5);
	\draw (-.5,.5)--(0,0)--(-.5,-.5);
	\node at (1,0) {$+$};
	\end{scope}
	\begin{scope}
	\draw (-.5,-.5)--(.5,.5);
	\draw (-.5,.5)--(.5,-.5);
	\draw [fill=black] (-.5,-.5) circle [radius=.04];
	\draw [fill=black] (.5,-.5) circle [radius=.04];
	\draw [fill=black] (-.5,.5) circle [radius=.04];
	\draw [fill=black] (.5,.5) circle [radius=.04];
	\draw (-.25,.25)--(.25,.25);
	\draw [fill=black] (-.25,.25) circle [radius=.04];
	\draw [fill=black] (.25,.25) circle [radius=.04];
	\draw [red] (0,.25) circle [radius=.04];
	\node at (1,0) {$+$};
	\end{scope}
	\begin{scope}[shift={(2,0)}]
	\draw (-.5,-.5)--(.5,.5);
	\draw (-.5,.5)--(.5,-.5);
	\draw [fill=black] (-.5,-.5) circle [radius=.04];
	\draw [fill=black] (.5,-.5) circle [radius=.04];
	\draw [fill=black] (-.5,.5) circle [radius=.04];
	\draw [fill=black] (.5,.5) circle [radius=.04];
	\draw [cyan] (-.25,.25)--(.25,.25);
	\draw [fill=black] (-.25,.25) circle [radius=.04];
	\draw [fill=black] (.25,.25) circle [radius=.04];
	\draw [cyan,ultra thick] (0,.3)--(0,.2);
	\end{scope}
	\end{tikzpicture}.
	\end{equation}
\end{proposition}
\begin{proof}
	It follows directly by computation. 
\end{proof}
\begin{notation}
	We denote the submodel of $\mathscr{P}_\bullet^{S_n}$ generated by $\mathscr{P}_2$ by $\mathscr{A}_\bullet$. 
\end{notation}
\begin{theorem}\label{lem: symmetric}
	The strongly regular graph $KG_{n,2}$ has property $(G)$ if and only if 
	
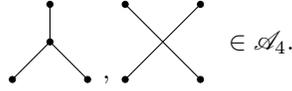
\begin{equation}\label{equ: keyS}
	\begin{tikzpicture}
	\begin{scope}[shift={(-1.5,0)}]
	\draw (0,0)--(0,.5);
	\draw (-.5,-.5)--(0,0)--(.5,-.5);
	\draw [fill=black] (0,0) circle [radius=.04];
	\draw [fill=black] (0,.5) circle [radius=.04];
	\draw [fill=black] (-.5,-.5) circle [radius=.04];
	\draw [fill=black] (.5,-.5) circle [radius=.04];
	\end{scope}
	\node at (-.75,-.5) {$,$};
	\draw (-.5,-.5)--(.5,.5);
	\draw (-.5,.5)--(.5,-.5);
	\draw [fill=black] (-.5,-.5) circle [radius=.04];
	\draw [fill=black] (.5,-.5) circle [radius=.04];
	\draw [fill=black] (-.5,.5) circle [radius=.04];
	\draw [fill=black] (.5,.5) circle [radius=.04];
	\node at (.75,0) [right] {$\in\mathscr{A}_4.$};
	\end{tikzpicture}
	\end{equation}
\end{theorem}
\begin{proof}
	Suppose $R\in\mathscr{A}_4$. We construct an element $X_n$ in $\mathscr{A}_{\binom{d}{2}}$ with respect to $KG_{n,2}$ as follows:
	\begin{enumerate}
		\item Draw $KG_{n,2}$ on $\mathbb{R}^2$ such that every vertex lies on the line $y=1$ and the remaining part of $KG_{n,2}$ are below the line $y=1$. Label the vertices from left to right by $1,2,\cdots,\binom{n}{2}$. 
		\item Each vertex is replaced by $\GHZ_{\binom{n-2}{2}}$; each edge is replaced by $A$; each crossing is replaced by $R$. This defines an element in $\mathscr{A}_{\binom{n}{2}}$, and we denote it by $X_n$. 
	\end{enumerate}
	
	Suppose $\vec{i}=(i_1,i_2,\cdots,i_{\binom{n}{2}})$ is an arbitrary point in $V^{\binom{n}{2}}$. By the construction, we have that $X_n(\vec{i})$ must take a value in $\{0,1\}$. Set $M=\{\vec{i}: X_n(\vec{i})=1\}\subset V^{\binom{n}{2}}$. Let $\vec{i}\in M$. First we show that for any $1\leq k,l\leq\binom{n}{2}$, we must have $i_k\neq i_l;$. This will be discussed in two cases by contradiction: Assume that there exists $1\leq k,l\leq \binom{n}{2}$ such that $i_k=i_l$. By definition of $A$, we know that $k$ and $l$ are not connected by an edge in $KG_{n,2}$.
	\begin{enumerate}
		\item When $n$ is odd:  By the construction of $KG_{n,2}$, there exists a subset $W\subset V$, such that
		\begin{itemize}
			\item The induced subgraph on $W$ is the complete graph on $\displaystyle\frac{n-1}{2}$ vertices. 
			\item There exists $a,b\in W$ such that $k$ is connected to every vertex in $W-\{b\}$ and $l$ is connected to every one in $W-\{a\}$. 
		\end{itemize}
		The assumption that $i_k=i_l$ implies that the induced subgraph on $\{i_k\}\sqcup\{i_w:w\in W\}$ is the complete graph on $\displaystyle\frac{n+1}{2}$ vertices. This leads to a contradiction.
		\item When $n$ is even: By the construction of $KG_{n,2}$, there exists $W\subset V$ and $c\in V$ such that
		\begin{itemize}
			\item The induced subgraph on $W\sqcup\{a\}$ is the complete graph on $\displaystyle\frac{n-2}2$ vertices. 
			\item There exists $a,b\in W$ such that $k$ is connected to every vertex in $W-\{b\}$ and $l$ is connected to every one in $W-\{a\}$. 
			\item The vertex $c$ is neither connected to $k$ nor $l$. 
		\end{itemize}
		Similarly, the assumption that $i_k=i_l$ implies that the induced subgraph on $\{i_k\}\sqcup\{i_w:w\in W\}$ is the complete graph on $\displaystyle\frac{n}2$ vertices.   Therefore, the two induced subgraphs on $\{i_k\}\sqcup\{i_w:w\in W\}$ and $\{i_c\}\sqcup\{i_w:w\in W\}$ are the complete graph on $n$ vertices and $i_k$ and $i_l$ are two different vertices. However, this is a contradiction by the Kneser construction of $KG_{n,2}$. 
	\end{enumerate}
	
	Therefore, for every $\vec{i}\in M$, we have that $i_j's$ are distinct. Let $g_{\vec{i}}$ be the permutation on $V$ defined by sending $i_j$ to $j$. By the construction of $X_n$, we know that $g_{\vec{i}}$ is an automorphism of $KG_{n,2}$. Moreover, for every automorphism $g$, we have that $X_n(\beta(g)1,\cdots,\beta(g)n)=1$. This implies that 
	\begin{equation}
	X_n=\chi_S=M.
	\end{equation}
	Since the assumption of this lemma asserts that $\GHZ$ and $R$ belongs to $\mathscr{A}_\bullet$, it follows from Proposition \ref{pro: diagrammatic representation} that
	\begin{equation}
	\mathscr{P}_\bullet^{S_n}\subset \mathscr{A}_\bullet\subset\mathscr{P}_\bullet^{S_n},
	\end{equation}  
	namely, the graph $KG_{n,2}$ has property $(G)$. 
	
	The other direction follows directly and therefore the lemma is proved.     
\end{proof}
\begin{lemma}\label{pro: reduce lemma}
	The strongly regular graph $KG_{n,2}$ has property $(G)$ if and only if
	\begin{equation}\label{equ: keySX}
	\begin{tikzpicture}
	\begin{scope}[shift={(-1.5,0)}]
	\draw (0,0)--(0,.5);
	\draw (-.5,-.5)--(0,0)--(.5,-.5);
	\draw [fill=black] (0,0) circle [radius=.04];
	\draw [fill=black] (0,.5) circle [radius=.04];
	\draw [fill=black] (-.5,-.5) circle [radius=.04];
	\draw [fill=black] (.5,-.5) circle [radius=.04];
	\end{scope}
	\node at (-.75,-.5) {$,$};
	\draw (-.5,-.5)--(.5,.5);
	\draw (-.5,.5)--(.5,-.5);
	\draw [fill=black] (-.5,-.5) circle [radius=.04];
	\draw [fill=black] (.5,-.5) circle [radius=.04];
	\draw [fill=black] (-.5,.5) circle [radius=.04];
	\draw [fill=black] (.5,.5) circle [radius=.04];
	\draw (-.25,.25)--(.25,.25);
	\draw [fill=black] (-.25,.25) circle [radius=.04];
	\draw [fill=black] (.25,.25) circle [radius=.04];
	\draw [red] (0,.25) circle [radius=.04];
	\node at (.75,0) [right] {$\in\mathscr{A}_4$};
	\end{tikzpicture}.
	\end{equation}
\end{lemma}
\begin{proof}
	$(\Rightarrow)$ It follows directly. 
	
	$(\Leftarrow)$ Suppose $R_A\in \mathscr{A}_\bullet$. By Lemma \ref{lem: symmetric} and Equation \ref{equ: decompostion of R}, we only need to show $R_T\in\mathscr{A}_4$. We define a sequence of elements $\gamma_k\in\mathscr{A}_{3k}$ recursively as follows:
	\begin{align}
	&\begin{tikzpicture}
	\draw (.5,.5)--(0,0)--(-.5,.5);
	\draw (0,0)--(0,-.5);
	\draw [fill=black] (0,0) circle[radius=.04];
	\node at (-.8,0) {$\gamma_1=$};
	\draw [fill=black] (-.5,.5) circle[radius=.04];
	\draw [fill=black] (0,-.5) circle[radius=.04];
	\draw [fill=black] (.5,.5) circle[radius=.04];
	\end{tikzpicture}.\\
	&\begin{tikzpicture}
	\node at (-1,0) {$\gamma_{k+1}=$};    
	\begin{scope}
	\draw (0,0)--(0,1);
	\draw (.4,0)--(.4,1);
	\draw (1,0)--(1,1);
	\draw (1.4,0)--(1.4,1);
	\draw [fill=black] (0,1) circle[radius=.04];
	\draw [fill=black] (.4,1) circle[radius=.04];
	\draw [fill=black] (1,1) circle[radius=.04];
	\draw [fill=black] (1.4,1) circle[radius=.04];
	\node at (.7,.2) {$\cdots$};
	\end{scope}
	\begin{scope}[shift={(2.2,0)}]
	\draw (0,0)--(0,1);
	\draw (.2,.6)--(0,.4);
	\draw (.4,0)--(.4,1);
	\draw (1,0)--(1,1);
	\draw (1.2,.6)--(1,.4);
	\draw (1.4,0)--(1.4,1);
	\draw (1.6,.6)--(1.4,.4);
	\node at (.7,.2) {$\cdots$};
	\draw [fill=black] (0,.4) circle [radius=.04];
	\draw [fill=black] (1,.4) circle [radius=.04];
	\draw [fill=black] (1.4,.4) circle [radius=.04];
	\draw [fill=black] (.2,.6) circle [radius=.04];
	\draw [fill=black] (1.2,.6) circle [radius=.04];
	\draw [fill=black] (1.6,.6) circle [radius=.04];
	\draw [red] (.1,.5) circle[radius=.04];
	\draw [red] (1.1,.5) circle[radius=.04];
	\draw [red] (1.5,.5) circle[radius=.04];
	\draw [fill=black] (0,1) circle[radius=.04];
	\draw [fill=black] (.4,1) circle[radius=.04];
	\draw [fill=black] (1,1) circle[radius=.04];
	\draw [fill=black] (1.4,1) circle[radius=.04];
	\end{scope}
	\begin{scope}[shift={(1.1,-1.5)}]
	\draw (0,0)--(0,1);
	\draw (.4,0)--(.4,1);
	\draw (1,0)--(1,1);
	\draw (1.4,0)--(1.4,1);
	\node at (.7,.7) {$\cdots$};
	\draw [fill=black] (0,0) circle[radius=.04];
	\draw [fill=black] (.4,0) circle[radius=.04];
	\draw [fill=black] (1,0) circle[radius=.04];
	\draw [fill=black] (1.4,0) circle[radius=.04];
	\end{scope}
	\draw (1.8,1)--(1.8,.8) arc[radius=.2, start angle=180, end angle=270]--(4,.6);
	\draw (4,1)--(4,-1.5);    
	\draw [fill=black] (4,.6) circle [radius=.04];
	\draw [fill=white] (-.1,0) rectangle (3.7,-.5);
	\node at (1.8,-.25) {$\gamma_k$};
	\node at (-.2,-.25) {$\$$};
	\draw [fill=black] (4,1) circle[radius=.04];
	\draw [fill=black] (4,-1.5) circle[radius=.04];
	\draw [fill=black] (1,1) circle[radius=.04];
	\end{tikzpicture}.
	\end{align}
	\begin{claim}
		Let $B_k=\{(i_1,i_2,\cdots,i_k,i_1,i_2,\cdots,i_k,i_k,i_{k-1},\cdots,i_1):S(i_s,i_t)=1~\forall 1\leq s\leq t\leq k\}$. Then we have that $\gamma_k=\chi_{B_k}$. 
	\end{claim}
	\begin{proof}[Proof of Claim]
		We prove this by induction on $k$. It is straightforward to see that the claim is true when $k=1$. Now assume it is true for $k$. Let $(i_1,i_2,\cdots, i_{k+1},j_1,j_2,\cdots,j_{k+1},m_{k+1},m_k,\cdots,m_1)\in V^{3k+3}$. We denote 
		\begin{align}
		&\gamma_{k+1}(i_1,i_2,\cdots, i_{k+1},j_1,j_2,\cdots,j_{k+1},m_{k+1},m_k,\cdots,m_1)\label{equ: indcution}\\
		=&\gamma_k(i_1,\cdots,i_k,j_1,\cdots,j_k,m_k,m_{k-1},\cdots,m_1)\delta_{i_{k+1},j_{k+1}}\delta_{i_{k+1},m_{k+1}}\prod_{t=1}^k S(i_{k+1},j_t)\\
		=&\left(\prod_{t=1}^k \delta_{i_t,m_t}\delta_{i_t,m_t}\right)\left(\prod_{1\leq s\leq t\leq k} S(i_s,i_t)\right)    \delta_{i_{k+1},j_{k+1}}\delta_{i_{k+1},m_{k+1}}\prod_{t=1}^k S(i_{k+1},j_t)\\
		=&\left(\prod_{t=1}^{k+1} \delta_{i_t,m_t}\delta_{i_t,m_t}\right)\left(\prod_{1\leq s\leq t\leq k+1} S(i_s,i_t)\right)\\
		=&\chi_{B_{k+1}}.
		\end{align}
		This proves the claim.
	\end{proof}
	Now we construct $R_T$ explicitly in the following two cases:
	\begin{enumerate}
		\item Suppose $n$ is odd. Let $m=\displaystyle\frac{n-3}2$. Consider the following element $Y$ in $\mathscr{A}_{4}$. 
		
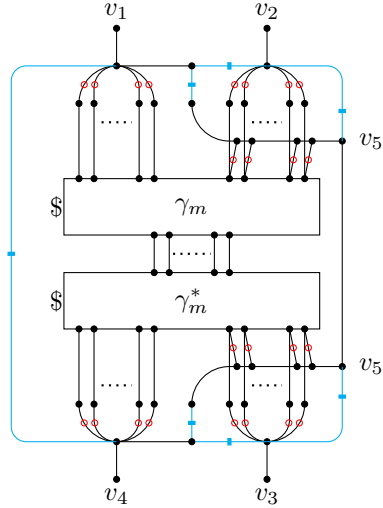
\begin{figure}[H]
			\begin{tikzpicture}
			\draw (.2,0)--(.2,1);
			\draw (.4,0)--(.4,1);
			\draw (1.2,0)--(1.2,1);
			\draw (1,0)--(1,1);
			\draw [fill=black] (.2,0) circle [radius=.04];
			\draw [fill=black] (.2,1) circle [radius=.04];
			\draw [fill=black] (.4,0) circle [radius=.04];
			\draw [fill=black] (.4,1) circle [radius=.04];
			\draw [fill=black] (1.2,0) circle [radius=.04];
			\draw [fill=black] (1.2,1) circle [radius=.04];
			\draw [fill=black] (1,0) circle [radius=.04];
			\draw [fill=black] (1,1) circle [radius=.04];
			\draw [dotted, thick] (.5,.75)--(.9,.75);
			\draw (.2,1) arc [radius=.5, start angle=180, end angle=0];
			\draw (.4,1)--(.4,1.2) arc[radius=.3, start angle=180, end angle=0]--(1,1);
			\draw [fill=black] (.7,1.5) circle [radius=.04];
			\draw [red] (.27,1.25) circle [radius=.04];
			\draw [red] (.41,1.25) circle [radius=.04];
			\draw [red] (.99,1.25) circle [radius=.04];
			\draw [red] (1.13,1.25) circle [radius=.04];
			\draw (2+.2,0)--(2+.2,1);
			\draw (2+.4,0)--(2+.4,1);
			\draw (2+1.2,0)--(2+1.2,1);
			\draw (2+1,0)--(2+1,1);
			\draw [fill=black] (2+.2,0) circle [radius=.04];
			\draw [fill=black] (2+.2,1) circle [radius=.04];
			\draw [fill=black] (2+.4,0) circle [radius=.04];
			\draw [fill=black] (2+.4,1) circle [radius=.04];
			\draw [fill=black] (2+1.2,0) circle [radius=.04];
			\draw [fill=black] (2+1.2,1) circle [radius=.04];
			\draw [fill=black] (2+1,0) circle [radius=.04];
			\draw [fill=black] (2+1,1) circle [radius=.04];
			\draw [dotted, thick] (2+.5,.75)--(2+.9,.75);
			\draw (2+.2,1) arc [radius=.5, start angle=180, end angle=0];
			\draw (2+.4,1)--(2+.4,1.2) arc[radius=.3, start angle=180, end angle=0]--(2+1,1);
			\draw [fill=black] (2+.7,1.5) circle [radius=.04];
			\draw [red] (2+.27,1.25) circle [radius=.04];
			\draw [red] (2+.41,1.25) circle [radius=.04];
			\draw [red] (2+.99,1.25) circle [radius=.04];
			\draw [red] (2+1.13,1.25) circle [radius=.04];
			\draw (.7,1.5)--(1.7,1.5);
			\draw [cyan](1.7,1.5)--(2.7,1.5);
			\draw [cyan](1.7,1.5)--(1.7,1);
			\draw [fill=black] (1.7,1) circle[radius=.04];
			\draw [fill=black] (1.7,1.5) circle [radius=.04];
			\draw [ultra thick, cyan] (1.65,1.25)--(1.75,1.25);
			\draw [ultra thick, cyan] (2.2,1.45)--(2.2,1.55);
			\draw (1.7,1) arc[radius=.5, start angle=180, end angle=270]--(3.7,.5);
			\draw [fill=black](2.3,.5) circle [radius=.04];
			\draw [fill=black] (2.5,.5) circle [radius=.04];
			\draw [fill=black] (3.1,.5) circle[radius=.04];
			\draw [fill=black] (3.3,.5) circle [radius=.04];
			\draw [fill=black] (3.7,.5) circle [radius=.04];
			\draw (2.2,0)--(2.3,.5);
			\draw [red] (2.25,.25) circle [radius=.04];
			\draw (2.4,0)--(2.5,.5);
			\draw [red] (2.45,.25) circle [radius=.04];
			\draw (3.1,.5)--(3,0);
			\draw [red] (3.05,.25) circle [radius=.04];
			\draw (3.2,0)--(3.3,.5);
			\draw [red] (3.25,.25) circle [radius=.04];
			\draw [cyan](2.7,1.5)--(3.5,1.5) arc[radius=.2, start angle=90, end angle=0]--(3.7,.5);
			\draw [ultra thick,cyan] (3.75,.9)--(3.65,.9);
			\draw (.2,-2-0)--(.2,-2-1);
			\draw (.4,-2-0)--(.4,-2-1);
			\draw (1.2,-2-0)--(1.2,-2-1);
			\draw (1,-2-0)--(1,-2-1);
			\draw [fill=black] (.2,-2-0) circle [radius=.04];
			\draw [fill=black] (.2,-2-1) circle [radius=.04];
			\draw [fill=black] (.4,-2-0) circle [radius=.04];
			\draw [fill=black] (.4,-2-1) circle [radius=.04];
			\draw [fill=black] (1.2,-2-0) circle [radius=.04];
			\draw [fill=black] (1.2,-2-1) circle [radius=.04];
			\draw [fill=black] (1,-2-0) circle [radius=.04];
			\draw [fill=black] (1,-2-1) circle [radius=.04];
			\draw [dotted,thick] (.5,-2-.75)--(.9,-2-.75);
			\draw (.2,-2-1) arc [radius=.5,start angle=180, end angle=360];
			\draw (.4,-2-1)--(.4,-2-1.2) arc[radius=.3,start angle=180,end angle=360]--(1,-2-1);
			\draw [fill=black] (.7,-2-1.5) circle [radius=.04];
			\draw [red] (.27,-2-1.25) circle [radius=.04];
			\draw [red] (.41,-2-1.25) circle [radius=.04];
			\draw [red] (.99,-2-1.25) circle [radius=.04];
			\draw [red] (1.13,-2-1.25) circle [radius=.04];
			\draw (2+.2,-2-0)--(2+.2,-2-1);
			\draw (2+.4,-2-0)--(2+.4,-2-1);
			\draw (2+1.2,-2-0)--(2+1.2,-2-1);
			\draw (2+1,-2-0)--(2+1,-2-1);
			\draw [fill=black] (2+.2,-2-0) circle [radius=.04];
			\draw [fill=black] (2+.2,-2-1) circle [radius=.04];
			\draw [fill=black] (2+.4,-2-0) circle [radius=.04];
			\draw [fill=black] (2+.4,-2-1) circle [radius=.04];
			\draw [fill=black] (2+1.2,-2-0) circle [radius=.04];
			\draw [fill=black] (2+1.2,-2-1) circle [radius=.04];
			\draw [fill=black] (2+1,-2-0) circle [radius=.04];
			\draw [fill=black] (2+1,-2-1) circle [radius=.04];
			\draw [dotted,thick] (2+.5,-2-.75)--(2+.9,-2-.75);
			\draw (2+.2,-2-1) arc [radius=.5,start angle=180,end angle=360];
			\draw (2+.4,-2-1)--(2+.4,-2-1.2) arc[radius=.3,start angle=180,end angle=360]--(2+1,-2-1);
			\draw [fill=black] (2+.7,-2-1.5) circle [radius=.04];
			\draw [red] (2+.27,-2-1.25) circle [radius=.04];
			\draw [red] (2+.41,-2-1.25) circle [radius=.04];
			\draw [red] (2+.99,-2-1.25) circle [radius=.04];
			\draw [red] (2+1.13,-2-1.25) circle [radius=.04];
			\draw (.7,-2-1.5)--(1.7,-2-1.5);
			\draw [cyan](1.7,-2-1.5)--(2.7,-2-1.5);
			\draw [cyan](1.7,-2-1.5)--(1.7,-2-1);
			\draw [fill=black] (1.7,-2-1) circle[radius=.04];
			\draw [fill=black] (1.7,-2-1.5) circle [radius=.04];
			\draw [ultra thick,cyan] (1.65,-2-1.25)--(1.75,-2-1.25);
			\draw [ultra thick,cyan] (2.2,-2-1.45)--(2.2,-2-1.55);
			\draw (1.7,-2-1) arc[radius=.5,start angle=180,end angle=90]--(3.7,-2-.5);
			\draw [fill=black](2.3,-2-.5) circle [radius=.04];
			\draw [fill=black] (2.5,-2-.5) circle [radius=.04];
			\draw [fill=black] (3.1,-2-.5) circle[radius=.04];
			\draw [fill=black] (3.3,-2-.5) circle [radius=.04];
			\draw [fill=black] (3.7,-2-.5) circle [radius=.04];
			\draw (2.2,-2-0)--(2.3,-2-.5);
			\draw [red] (2.25,-2-.25) circle [radius=.04];
			\draw (2.4,-2-0)--(2.5,-2-.5);
			\draw [red] (2.45,-2-.25) circle [radius=.04];
			\draw (3.1,-2-.5)--(3,-2-0);
			\draw [red] (3.05,-2-.25) circle [radius=.04];
			\draw (3.2,-2-0)--(3.3,-2-.5);
			\draw [red] (3.25,-2-.25) circle [radius=.04];
			\draw [cyan](2.7,-2-1.5)--(3.5,-2-1.5) arc[radius=.2,start angle=270,end angle=360]--(3.7,-2-.5);
			\draw [ultra thick,cyan] (3.75,-2-.9)--(3.65,-2-.9);
			\draw  (0,0) rectangle (3.4,-.75);
			\node at (-.1,-.375) {$\$$};
			\draw  (0,-2) rectangle (3.4,-1.25);
			\node at (-.1,-1.625) {$\$$};
			\node at (1.7,-0.375) {$\gamma_m$};
			\node at (1.7,-1.25-0.375) {$\gamma_m^*$};
			\draw (1.4,-.75)--(1.4,-1.25);
			\draw (1.2,-.75)--(1.2,-1.25);
			\draw (2,-.75)--(2,-1.25);
			\draw (2.2,-.75)--(2.2,-1.25);
			\draw [fill=black] (1.4,-.75) circle [radius=.04];
			\draw [fill=black] (1.2,-.75) circle [radius=.04];
			\draw [fill=black] (2.2,-.75) circle [radius=.04];
			\draw [fill=black] (2,-.75) circle [radius=.04];
			\draw [fill=black] (2,-1.25) circle [radius=.04];
			\draw [fill=black] (2.2,-1.25) circle [radius=.04];
			\draw [fill=black] (1.2,-1.25) circle [radius=.04];
			\draw [fill=black] (1.4,-1.25) circle [radius=.04];
			\draw [thick, dotted] (1.45,-1)--(1.95,-1);
			\draw [cyan] (.7,1.5)--(-.5,1.5) arc [radius=.2, start angle=90, end angle=180]--(-.7,-2-1.3) arc[radius=.2, start angle=180, end angle=270]--(.7,-2-1.5);
			\draw [ultra thick, cyan](-.65,-1)--(-.75,-1);
			\draw (3.7,.5)--(3.7,-2-.5);
			\draw (.7,1.5)--(.7,2);
			\draw (2.7,1.5)--(2.7,2);
			\draw (.7,-2-1.5)--(.7,-2-2);
			\draw (2.7,-2-1.5)--(2.7,-2-2);
			\draw [fill=black] (.7,2) circle[radius=.04];
			\draw [fill=black] (2.7,2) circle[radius=.04];
			\draw [fill=black] (2.7,-4) circle[radius=.04];
			\draw [fill=black] (.7,-4) circle[radius=.04];
			\node at (3.8,.5) [right] {$v_5$};
			\node at (3.8,-2-.5) [right] {$v_5$};
			\node at (.7,2) [above] {$v_1$};
			\node at (2.7,2) [above] {$v_2$};
			\node at (2.7,-4) [below] {$v_3$};
			\node at (.7,-4) [below] {$v_4$};
			\end{tikzpicture}
		\end{figure}
		Now we evaluate $Y$ on an arbitrary point $(v_1,v_2,v_3,v_4)\in V^4$. By definition, we have that 
		\begin{align}\label{equ: Eva of Y odd}
		\displaystyle\frac{Y(v_1,v_2,v_3,v_4)}{T(v_1,v_2)T(v_3,v_4)T(v_1,v_4)}=\sum_{v_5\in V;\vec{i}\in V^m}\left(\prod_{t=1}^4 T(v_t,v_5)\right)\left(\prod_{1\leq s\leq m,1\leq t\leq 5} S(i_s,v_t)  \right).
		\end{align}
		Note that $Y$ is invariant under the action $\beta$ of $S_n$. One can assume $v_1=\{1,2\}$ and $v_2=\{1,3\}$ without loss of generality. By definition of $\gamma_m$, we know that 
		\begin{equation}
		\bigsqcup_{t=1}^m i_t= \{4,5,\cdots,n\}. 
		\end{equation} 
		This implies that $v_5=\{2,3\}$ and $v_4,v_3\subset \{1,2,3\}$. Since there are terms $T(v_3,v_5)$, $T(v_4,v_5)$ and $T(v_1,v_2)T(v_3,v_4)T(v_1,v_4)$, we know that 
		$v_3=\{1,2\}$ and $v_4=\{1,3\}$ if $Y(v_1,v_2,v_3,v_4)\neq 0$. Moreover, by counting the possible choice of $\vec{i}$, we know that 
		\begin{equation}
		Y(v_1,v_2,v_3,v_4)=\frac{(n-3)!}{2^m} \delta_{v_1,v_3}\delta_{v_2,v_4} T(v_1,v_2),
		\end{equation}
		namely, $Y=\displaystyle\frac{(n-3)!}{2^m} R_T$. 
		\item Suppose $n$ is even. Let $m=\displaystyle\frac{n-4}{2}$. Consider the following element $Y$ in $\mathscr{A}_4$.  
		
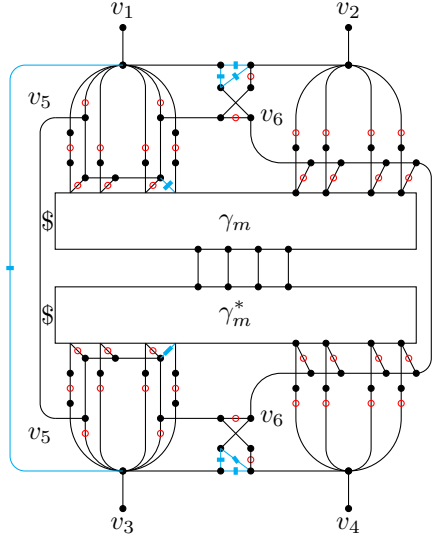
\begin{figure}[H]
			\begin{tikzpicture}
			\draw (0,1)--(0,2) arc[radius=.7, start angle=180, end angle=0]--(1.4,1);
			\draw (.4,1)--(.4,2.4) arc[radius=.3, start angle=180, end angle=0]--(1,1);
			\draw (.2,2)--(.2,2.2) arc[radius=.5, start angle=180, end angle=0]--(1.2,2);
			\draw [red](.2,2.2) circle [radius=.04];
			\draw [red](1.2,2.2) circle [radius=.04];
			\draw (.2,2)--(-.2,2);
			\draw (.2,2)--(.2,1.2)--(1.2,1.2)--(1.2,2);
			\draw [red] (0,1.6) circle [radius=.04];
			\draw [red] (.4,1.6) circle [radius=.04];
			\draw [red] (1,1.6) circle [radius=.04];
			\draw [red] (1.4,1.6) circle [radius=.04];
			\draw [fill=black] (.7,2.7) circle [radius=.04];
			\draw [fill=black] (0,1.8) circle [radius=.04];
			\draw [fill=black] (1.4,1.8) circle [radius=.04];
			\draw [fill=black] (0,1.4) circle [radius=.04];
			\draw [fill=black] (.4,1.4) circle [radius=.04];
			\draw [fill=black] (1,1.4) circle [radius=.04];
			\draw [fill=black] (1.4,1.4) circle [radius=.04];
			\draw [fill=black] (.2,2) circle [radius=.04];
			\draw [fill=black] (1.2,2) circle [radius=.04];
			\draw [fill=black] (.2,1.2) circle [radius=.04];
			\draw [fill=black] (1.2,1.2) circle [radius=.04];
			\draw [fill=black] (.6,1.2) circle [radius=.04];
			\draw (.2,1.2)--(0,1);
			\draw (.6,1.2)--(.4,1);
			\draw (1.2,1.2)--(1,1);
			\draw [red] (.1,1.1) circle [radius=.04];
			\draw [red] (.5,1.1) circle [radius=.04];
			\draw [red] (1.1,1.1) circle [radius=.04];
			\draw [cyan](1.2,1.2)--(1.4,1);
			\draw [cyan, ultra thick] (1.3-.05,1.1-.05)--(1.3+.05,1.1+.05);
			\draw (1.2,2)--(2.4,2);
			\draw [fill=black](2,2) circle [radius=.04];
			\draw [fill=black] (2.4,2) circle [radius=.04];
			\draw [red] (2.2,2) circle [radius=.04];
			\draw [fill=black](2,2.4) circle [radius=.04];
			\draw [fill=black] (2.4,2.4) circle [radius=.04];
			\draw (2,2.4)--(2.4,2);
			\draw (2.4,2.4)--(2,2);
			\draw [cyan] (2,2.4)--(2.4,2.7);
			\draw (.7,2.7)--(2,2.7);
			\draw [cyan] (2,2.4)--(2,2.7)--(2.4,2.7);
			\draw [ultra thick, cyan] (1.95,2.55)--(2.05,2.55);
			\draw [ultra thick, cyan] (2.2,2.65)--(2.2,2.75);
			\draw [ultra thick, cyan] (2.2-.03,2.55+.04)--(2.2+.03,2.55-.04);
			\draw [red] (2.4,2.55) circle [radius=.04];
			\draw (2.4,2.7)--(2.4,2.4);
			\draw [fill=black] (2.4,2.7)circle [radius=.04];
			\draw [fill=black] (2,2.7)circle [radius=.04];
			\draw (2.4,2.7)--(3.7,2.7);
			\draw [fill=black] (3.7,2.7) circle [radius=.04];
			\draw (3+0,1)--(3+0,2) arc[radius=.7, start angle=180, end angle=0]--(3+1.4,1);
			\draw (3+.4,1)--(3+.4,2.4) arc[radius=.3, start angle=180, end angle=0]--(3+1,1);
			\draw (2.4,2)--(2.4,1.8) arc[radius=.4,start angle=180, end angle=270]--(4.6,1.4);
			\draw (3.2,1.4)--(3,1);
			\draw (3.6,1.4)--(3.4,1);
			\draw (4.2,1.4)--(4,1);
			\draw (4.6,1.4)--(4.4,1);
			\draw [fill=black](3.2,1.4) circle [radius=.04];
			\draw [fill=black](3.6,1.4) circle [radius=.04];
			\draw [fill=black](4.2,1.4) circle [radius=.04];
			\draw [fill=black](4.6,1.4) circle [radius=.04];
			\draw [fill=black](3,1) circle [radius=.04];
			\draw [fill=black](3.4,1) circle [radius=.04];
			\draw [fill=black](4,1) circle [radius=.04];
			\draw [fill=black](4.4,1) circle [radius=.04];
			\draw [fill=black](3,1.6) circle [radius=.04];
			\draw [fill=black](3.4,1.6) circle [radius=.04];
			\draw [fill=black](4,1.6) circle [radius=.04];
			\draw [fill=black](4.4,1.6) circle [radius=.04];
			\draw [red] (3.1,1.2) circle [radius=.04];
			\draw [red] (3.5,1.2) circle [radius=.04];
			\draw [red] (4.1,1.2) circle [radius=.04];
			\draw [red] (4.5,1.2) circle [radius=.04];
			\draw [red] (3,1.8) circle [radius=.04];
			\draw [red] (3.4,1.8) circle [radius=.04];
			\draw [red] (4,1.8) circle [radius=.04];
			\draw [red] (4.4,1.8) circle [radius=.04];
			\draw (0,-1)--(0,-2) arc[radius=.7,start angle=180,end angle=360]--(1.4,-1);
			\draw (.4,-1)--(.4,-2.4) arc[radius=.3,start angle=180,end angle=360]--(1,-1);
			\draw (.2,-2)--(.2,-2.2) arc[radius=.5,start angle=180,end angle=360]--(1.2,-2);
			\draw [red](.2,-2.2) circle [radius=.04];
			\draw [red](1.2,-2.2) circle [radius=.04];
			\draw (.2,-2)--(-.2,-2);
			\draw (.2,-2)--(.2,-1.2)--(1.2,-1.2)--(1.2,-2);
			\draw [red] (0,-1.6) circle [radius=.04];
			\draw [red] (.4,-1.6) circle [radius=.04];
			\draw [red] (1,-1.6) circle [radius=.04];
			\draw [red] (1.4,-1.6) circle [radius=.04];
			\draw [fill=black] (.7,-2.7) circle [radius=.04];
			\draw [fill=black] (0,-1.8) circle [radius=.04];
			\draw [fill=black] (1.4,-1.8) circle [radius=.04];
			\draw [fill=black] (0,-1.4) circle [radius=.04];
			\draw [fill=black] (.4,-1.4) circle [radius=.04];
			\draw [fill=black] (1,-1.4) circle [radius=.04];
			\draw [fill=black] (1.4,-1.4) circle [radius=.04];
			\draw [fill=black] (.2,-2) circle [radius=.04];
			\draw [fill=black] (1.2,-2) circle [radius=.04];
			\draw [fill=black] (.2,-1.2) circle [radius=.04];
			\draw [fill=black] (1.2,-1.2) circle [radius=.04];
			\draw [fill=black] (.6,-1.2) circle [radius=.04];
			\draw (.2,-1.2)--(0,-1);
			\draw (.6,-1.2)--(.4,-1);
			\draw (1.2,-1.2)--(1,-1);
			\draw [red] (.1,-1.1) circle [radius=.04];
			\draw [red] (.5,-1.1) circle [radius=.04];
			\draw [red] (1.1,-1.1) circle [radius=.04];
			\draw [cyan](1.2,-1.2)--(1.4,-1);
			\draw [cyan, ultra thick] (1.3-.05,-1.1-.05)--(1.3+.05,-1.1+.05);
			\draw (1.2,-2)--(2.4,-2);
			\draw [fill=black](2,-2) circle [radius=.04];
			\draw [fill=black] (2.4,-2) circle [radius=.04];
			\draw [red] (2.2,-2) circle [radius=.04];
			\draw [fill=black](2,-2.4) circle [radius=.04];
			\draw [fill=black] (2.4,-2.4) circle [radius=.04];
			\draw (2,-2.4)--(2.4,-2);
			\draw (2.4,-2.4)--(2,-2);
			\draw [cyan] (2,-2.4)--(2.4,-2.7);
			\draw (.7,-2.7)--(2,-2.7);
			\draw [cyan] (2,-2.4)--(2,-2.7)--(2.4,-2.7);
			\draw [ultra thick,cyan] (1.95,-2.55)--(2.05,-2.55);
			\draw [ultra thick,cyan] (2.2,-2.65)--(2.2,-2.75);
			\draw [ultra thick,cyan] (2.2-.03,-2.55+.04)--(2.2+.03,-2.55-.04);
			\draw [red] (2.4,-2.55) circle [radius=.04];
			\draw (2.4,-2.7)--(2.4,-2.4);
			\draw [fill=black] (2.4,-2.7)circle [radius=.04];
			\draw [fill=black] (2,-2.7)circle [radius=.04];
			\draw (2.4,-2.7)--(3.7,-2.7);
			\draw [fill=black] (3.7,-2.7) circle [radius=.04];
			\draw (3+0,-1)--(3+0,-2) arc[radius=.7,start angle=180,end angle=360]--(3+1.4,-1);
			\draw (3+.4,-1)--(3+.4,-2.4) arc[radius=.3,start angle=180,end angle=360]--(3+1,-1);
			\draw (2.4,-2)--(2.4,-1.8) arc[radius=.4,start angle=180,end angle=90]--(4.6,-1.4);
			\draw (3.2,-1.4)--(3,-1);
			\draw (3.6,-1.4)--(3.4,-1);
			\draw (4.2,-1.4)--(4,-1);
			\draw (4.6,-1.4)--(4.4,-1);
			\draw [fill=black](3.2,-1.4) circle [radius=.04];
			\draw [fill=black](3.6,-1.4) circle [radius=.04];
			\draw [fill=black](4.2,-1.4) circle [radius=.04];
			\draw [fill=black](4.6,-1.4) circle [radius=.04];
			\draw [fill=black](3,-1) circle [radius=.04];
			\draw [fill=black](3.4,-1) circle [radius=.04];
			\draw [fill=black](4,-1) circle [radius=.04];
			\draw [fill=black](4.4,-1) circle [radius=.04];
			\draw [fill=black](3,-1.6) circle [radius=.04];
			\draw [fill=black](3.4,-1.6) circle [radius=.04];
			\draw [fill=black](4,-1.6) circle [radius=.04];
			\draw [fill=black](4.4,-1.6) circle [radius=.04];
			\draw [red] (3.1,-1.2) circle [radius=.04];
			\draw [red] (3.5,-1.2) circle [radius=.04];
			\draw [red] (4.1,-1.2) circle [radius=.04];
			\draw [red] (4.5,-1.2) circle [radius=.04];
			\draw [red] (3,-1.8) circle [radius=.04];
			\draw [red] (3.4,-1.8) circle [radius=.04];
			\draw [red] (4,-1.8) circle [radius=.04];
			\draw [red] (4.4,-1.8) circle [radius=.04];
			\draw (-.2,1) rectangle (4.6,.25);
			\node at (-.3,.625) {$\$$};
			\draw (-.2,-1) rectangle (4.6,-.25);
			\node at (-.3,-.625) {$\$$};
			\node at (2.2,.6) {$\gamma_m$};
			\node at (2.2,-.6) {$\gamma_m^*$};
			\draw [fill=black] (2.5,.25) circle [radius=.04];
			\draw [fill=black](2.9,.25) circle [radius=.04];
			\draw [fill=black] (2.1,.25) circle [radius=.04]; 
			\draw [fill=black] (1.7,.25) circle [radius=.04];
			\draw [fill=black] (2.5,-.25) circle [radius=.04];
			\draw [fill=black](2.9,-.25) circle [radius=.04];
			\draw [fill=black] (2.1,-.25) circle [radius=.04]; 
			\draw [fill=black] (1.7,-.25) circle [radius=.04];
			\draw (2.5,.25)--(2.5,-.25);
			\draw (2.9,.25)--(2.9,-.25);
			\draw (2.1,.25)--(2.1,-.25);
			\draw (1.7,.25)--(1.7,-.25);
			\draw (-.2,2) arc[radius=.2, start angle=90, end angle=180]--(-.4,-1.8) arc[radius=.2, start angle=180, end angle=270];
			\draw (4.6,1.4) arc[radius=.2, start angle=90, end angle=0]--(4.8,-1.2) arc[radius=.2, start angle=0, end angle=-90];
			\draw (.7,2.7)--(.7,3.2);
			\draw (3.7,2.7)--(3.7,3.2);
			\draw (.7,-2.7)--(.7,-3.2);
			\draw (3.7,-2.7)--(3.7,-3.2);
			\draw [fill=black](.7,2.7) circle [radius=.04];
			\draw [fill=black](.7,3.2) circle [radius=.04];
			\draw [fill=black](3.7,2.7)circle [radius=.04];
			\draw [fill=black](3.7,3.2)circle [radius=.04];;
			\draw [fill=black](.7,-2.7)circle [radius=.04];
			\draw [fill=black](.7,-3.2)circle [radius=.04];
			\draw [fill=black](3.7,-2.7)circle [radius=.04];
			\draw [fill=black](3.7,-3.2)circle [radius=.04];
			\node at (.7,3.2) [above] {$v_1$};
			\node at (3.7,3.2) [above] {$v_2$};
			\node at (.7,-3.2) [below] {$v_3$};
			\node at (3.7,-3.2) [below] {$v_4$};
			\node at (-.4,2) [left,above] {$v_5$};
			\node at (-.4,-2)[right,below] {$v_5$};
			\node at (2.4,2) [right] {$v_6$};
			\node at (2.4,-2) [right] {$v_6$};
			\draw [cyan] (.7,2.7)--(-.6,2.7) arc[radius=.2, start angle=90, end angle=180]--(-.8,-2.5) arc[radius=.2, start angle=180, end angle=270]--(.7,-2.7);
			\draw [ultra thick, cyan] (-.85,0)--(-.75,0);
			\end{tikzpicture}
		\end{figure}
		Now we evaluate $Y$ on an arbitrary point $(v_1,v_2,v_3,v_4)\in V^4$. By definition, we have that
		\begin{equation}
		\displaystyle\frac{Y(v_1,v_2,v_3,v_4)}{T(v_1,v_2)T(v_3,v_4)T(v_1,v_4)}=\sum_{v_5,v_6\in V;\vec{i}\in V^m}\left(\prod_{j=1}^2S(v_5,v_j)T(v_6,v_j)\right)\displaystyle\frac{S(v_5,v_6)T(v_5,i_m)}{S(v_5,i_m)}\left(\prod_{t=1}^{m}S(v_5,i_t)S(v_6,i_t)\right)
		\end{equation}
		Without loss of generality, one can assume that $v_1=\{1,2\}$ and $v_2=\{1,3\}$. By definition of $\gamma_m$, we know that 
		\begin{equation}
		\left(\displaystyle\bigsqcup_{t=1}^m i_t\right)\bigcup v_5=\{4,5,\cdots,n\}.
		\end{equation}
		This implies that $v_6=\{2,3\}$. The rest of the computation is exactly similar to that in the previous case and thus $Y$ is a multiple of $R_T$. To be more precise, we have that
		\begin{equation}
		Y=\displaystyle\frac{(n-3)!}{2^{m-1}}\delta_{v_1,v_3}\delta_{v_2,v_4} T(v_1,v_2)=\displaystyle\frac{(n-3)!}{2^{m-1}}R_T. 
		\end{equation}
		
	\end{enumerate}
	Therefore, in both cases we have that
	\begin{equation}\label{equ: keySTeven}
	\begin{tikzpicture}
	\draw (0,0)--(1,1);
	\draw (0,1)--(1,0);
	\draw [fill=black] (0,0) circle[radius=.04];
	\draw [fill=black] (1,1) circle[radius=.04];
	\draw [fill=black] (1,0) circle[radius=.04];
	\draw [fill=black] (0,1) circle[radius=.04];
	\draw [cyan](.25,.75)--(.75,.75);
	\draw [ultra thick, cyan] (.5,.7)--(.5,.8); 
	\draw [fill=black](.25,.75) circle [radius=.04];
	\draw [fill=black](.75,.75) circle [radius=.04];
	\node at (1.75,.5) {$\in\mathscr{A}_{4,+}$};        
	\end{tikzpicture}.
	\end{equation}
	Combined with Equation \eqref{equ: keySX}, we know that the transposition $R$ belongs to $\mathscr{A}_\bullet$, since $R=R_T+R_A+\GHZ_4$. By Theorem \ref{lem: symmetric}, we have that $\mathscr{A}_\bullet=\mathscr{P}_\bullet^{S_n}$, namely, $KG_{n,2}$ has property $(G)$. 
\end{proof}
Now we return to prove the main theorem. 
\begin{theorem}
	The Kneser graph $KG_{n,2}$ has property $(G)$ for $n\geq 5$. 
\end{theorem}
\begin{proof}
	By Lemma \ref{pro: reduce lemma}, we need to show that 
	\begin{equation}
	\begin{tikzpicture}
	\draw (0,0)--(1,1);
	\draw (0,1)--(1,0);
	\draw (.25,.75)--(.75,.75);
	\draw [red] (.5,.75) circle [radius=.04];
	\draw [fill=black](.25,.75) circle [radius=.04];
	\draw [fill=black](.75,.75) circle [radius=.04];
	\node at (1.75,.5) {$\in\mathscr{A}_{4,+}$};
	\end{tikzpicture}.
	\end{equation}
	To show this, we investigate the subspace $Q\subset\mathscr{P}_4^{S_n}$ defined by the fixed points under the following tangle:
	
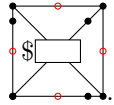
\begin{figure}[H]
		\begin{tikzpicture}
		\draw (0,0) rectangle (1.2,1.2);
		\draw [fill=black] (0,0) circle [radius=.04];
		\draw [fill=black] (0,1.2) circle [radius=.04];
		\draw [fill=black] (1.2,1.2) circle [radius=.04];
		\draw [fill=black] (1.2,0) circle [radius=.04];
		\draw (0,0)--(1.2,1.2);
		\draw (0,1.2)--(1.2,0);
		\draw [fill=white] (.3,.45) rectangle (.9,.75);
		\node at (.2,.6) {$\$$};
		\draw [red] (0,.6) circle[radius=.04];
		\draw [red] (1.2,.6) circle[radius=.04];
		\draw [red] (.6,0) circle[radius=.04];
		\draw [red] (.6,1.2) circle[radius=.04];
		\draw [fill=black] (0,0) circle[radius=.04];
		\draw [fill=black] (1,1) circle[radius=.04];
		\draw [fill=black] (1,0) circle[radius=.04];
		\draw [fill=black] (0,1) circle[radius=.04];
		\end{tikzpicture}.
	\end{figure}
	The idea is to show that $\mathscr{Q}\subset\mathscr{A}_{4}$. This implies that $R_A\in\mathscr{A}_4$ and by Lemma \ref{pro: reduce lemma}, we have that the transposition $R\in\mathscr{A}_4$. We discuss this in two cases:
	\begin{enumerate}
		\item \textbf{General case:} Suppose $n\geq 8$. It follows that the subspace $Q$ is $9$-dimensional and has the following basis $\mathcal{B}$:
		\begin{align*}
		b_1&=[\{1,2\},\{3,4\},\{1,2\},\{3,4\}],\\
		b_2&=[\{1,2\},\{3,4\},\{1,2\},\{3,5\}],\\
		b_3&=[\{1,2\},\{3,4\},\{1,2\},\{5,6\}],\\
		b_4&=[\{1,2\},\{3,4\},\{1,5\},\{3,4\}],\\
		b_5&=[\{1,2\},\{3,4\},\{1,5\},\{3,6\}],\\
		b_6&=[\{1,2\},\{3,4\},\{1,5\},\{6,7\}],\\
		b_7&=[\{1,2\},\{3,4\},\{5,6\},\{3,4\}],\\
		b_8&=[\{1,2\},\{3,4\},\{5,6\},\{3,7\}],\\
		b_9&=[\{1,2\},\{3,4\},\{5,6\},\{7,8\}].
		\end{align*}
		Let $\mathscr{B}$ be the set of following elements in $\mathscr{A}_4\cap\mathscr{Q}$:
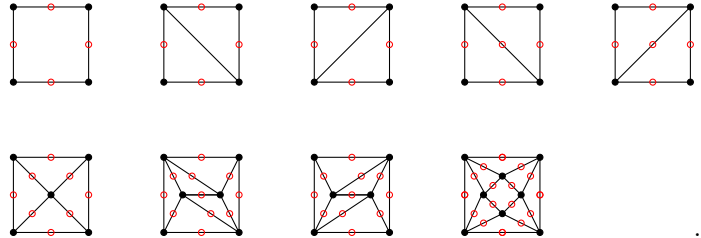
		\begin{figure}[H]\label{fig: diagrams}
			\begin{tikzpicture}
			\draw (1,1)--(0,0);
			\draw (0,1)--(1,-0);
			\draw [fill=white] (0,0)rectangle (1,1);
			\draw [red] (.5,1) circle [radius=.04];
			\draw [red] (0,.5) circle [radius=.04];
			\draw [red] (.5,0) circle [radius=.04];
			\draw [red] (1,.5) circle [radius=.04];
			\draw [fill=black] (0,0) circle [radius=.04];
			\draw [fill=black] (1,0) circle [radius=.04];
			\draw [fill=black] (0,1) circle [radius=.04];
			\draw [fill=black] (1,1) circle [radius=.04];
			\draw (2+1,1)--(2,0);    
			\draw [fill=white] (2+0,0) rectangle (2+1,1);
			\draw (2+0,1)--(2+1,0);
			\draw [red] (2+.5,1) circle [radius=.04];
			\draw [red] (2+0,.5) circle [radius=.04];
			\draw [red] (2+.5,0) circle [radius=.04];
			\draw [red] (2+1,.5) circle [radius=.04];
			\draw [fill=black] (2+0,0) circle [radius=.04];
			\draw [fill=black] (2+1,0) circle [radius=.04];
			\draw [fill=black] (2+0,1) circle [radius=.04];
			\draw [fill=black] (2+1,1) circle [radius=.04];
			\draw (4,1)--(4+1,0);
			\draw [fill=white] (4+0,0) rectangle (4+1,1);
			\draw (4+1,1)--(4,0);    
			\draw [red] (4+.5,1) circle [radius=.04];
			\draw [red] (4+0,.5) circle [radius=.04];
			\draw [red] (4+.5,0) circle [radius=.04];
			\draw [red] (4+1,.5) circle [radius=.04];
			\draw [fill=black] (4+0,0) circle [radius=.04];
			\draw [fill=black] (4+1,0) circle [radius=.04];
			\draw [fill=black] (4+0,1) circle [radius=.04];
			\draw [fill=black] (4+1,1) circle [radius=.04];
			\draw (6+1,1)--(6,0);    
			\draw [fill=white] (6+0,0) rectangle (6+1,1);
			\draw (6,1)--(6+1,0);
			\draw [red] (6+.5,1) circle [radius=.04];
			\draw [red] (6+0,.5) circle [radius=.04];
			\draw [red] (6+.5,0) circle [radius=.04];
			\draw [red] (6+1,.5) circle [radius=.04];
			\draw [fill=black] (6+0,0) circle [radius=.04];
			\draw [fill=black] (6+1,0) circle [radius=.04];
			\draw [fill=black] (6+0,1) circle [radius=.04];
			\draw [fill=black] (6+1,1) circle [radius=.04];
			\draw [red] (6+.5,.5) circle [radius=.04];
			\draw (8,1)--(8+1,0);
			\draw [fill=white] (8+0,0) rectangle (8+1,1);
			\draw (8+1,1)--(8,0);    
			\draw [red] (8+.5,1) circle [radius=.04];
			\draw [red] (8+0,.5) circle [radius=.04];
			\draw [red] (8+.5,0) circle [radius=.04];
			\draw [red] (8+1,.5) circle [radius=.04];
			\draw [fill=black] (8+0,0) circle [radius=.04];
			\draw [fill=black] (8+1,0) circle [radius=.04];
			\draw [fill=black] (8+0,1) circle [radius=.04];
			\draw [fill=black] (8+1,1) circle [radius=.04];
			\draw [red] (8+.5,.5) circle [radius=.04];
			\draw [fill=white] (0,0-2)rectangle (1,1-2);
			\draw (1,1-2)--(0,-2);
			\draw (0,1-2)--(1,-2);
			\draw [fill=black](.5,.5-2) circle [radius=.04];
			\draw [red] (.25,.75-2) circle [radius=.04];
			\draw [red] (.75,.75-2) circle [radius=.04];
			\draw [red] (.75,.25-2) circle [radius=.04];
			\draw [red] (.25,.25-2) circle [radius=.04];
			\draw [red] (.5,1-2) circle [radius=.04];
			\draw [red] (0,.5-2) circle [radius=.04];
			\draw [red] (.5,0-2) circle [radius=.04];
			\draw [red] (1,.5-2) circle [radius=.04];
			\draw [fill=black] (0,0-2) circle [radius=.04];
			\draw [fill=black] (1,0-2) circle [radius=.04];
			\draw [fill=black] (0,1-2) circle [radius=.04];
			\draw [fill=black] (1,1-2) circle [radius=.04];
			\draw (2+1,1-2)--(2,-2);
			\draw (2,1-2)--(2+1,-2);
			\draw [fill=white] (2+0,0-2)rectangle (2+1,1-2);
			\draw (2+0,1-2)--(2+.25,.5-2)--(2+.75,.5-2)--(2+1,1-2);
			\draw (2+0,0-2)--(2+.25,.5-2)--(2+.75,.5-2)--(2+1,0-2);
			\draw (2+0,1-2)--(2+.75,.5-2);
			\draw [red] (2+.375,.75-2) circle [radius=.04];
			\draw (2+1,0-2)--(2+.25,.5-2);
			\draw [red] (2+.625,.25-2) circle [radius=.04];
			\draw [red] (2+.125,.75-2) circle [radius=.04];
			\draw [red] (2+.125,.25-2) circle [radius=.04];
			\draw [red] (2+.5,.5-2) circle [radius=.04];
			\draw [red] (2+.875,.75-2) circle [radius=.04];
			\draw [red] (2+.875,.25-2) circle [radius=.04];
			\draw [fill=black](2+.25,.5-2) circle [radius=.04];
			\draw [fill=black](2+.75,.5-2) circle [radius=.04];
			\draw [fill=black] (2+0,0-2) circle [radius=.04];
			\draw [fill=black] (2+1,0-2) circle [radius=.04];
			\draw [fill=black] (2+0,1-2) circle [radius=.04];
			\draw [fill=black] (2+1,1-2) circle [radius=.04];
			\draw [red] (2+.5,1-2) circle [radius=.04];
			\draw [red] (2+.5,0-2) circle [radius=.04];
			\draw [red] (2+0,.5-2) circle [radius=.04];
			\draw [red] (2+1,.5-2) circle [radius=.04];
			\draw (4+1,1-2)--(4,-2);
			\draw (4,1-2)--(4+1,-2);
			\draw [fill=white] (4+0,0-2)rectangle (4+1,1-2);
			\draw (4+0,1-2)--(4+.25,.5-2)--(4+.75,.5-2)--(4+1,1-2);
			\draw (4+0,0-2)--(4+.25,.5-2)--(4+.75,.5-2)--(4+1,0-2);
			\draw (4+0,0-2)--(4+.75,.5-2);
			\draw [red] (4+.625,.75-2) circle [radius=.04];
			\draw (4+1,1-2)--(4+.25,.5-2);
			\draw [red] (4+.375,.25-2) circle [radius=.04];
			\draw [red] (4+.125,.75-2) circle [radius=.04];
			\draw [red] (4+.125,.25-2) circle [radius=.04];
			\draw [red] (4+.5,.5-2) circle [radius=.04];
			\draw [red] (4+.875,.75-2) circle [radius=.04];
			\draw [red] (4+.875,.25-2) circle [radius=.04];
			\draw [fill=black](4+.25,.5-2) circle [radius=.04];
			\draw [fill=black](4+.75,.5-2) circle [radius=.04];
			\draw [fill=black] (4+0,0-2) circle [radius=.04];
			\draw [fill=black] (4+1,0-2) circle [radius=.04];
			\draw [fill=black] (4+0,1-2) circle [radius=.04];
			\draw [fill=black] (4+1,1-2) circle [radius=.04];
			\draw [red] (4+.5,1-2) circle [radius=.04];
			\draw [red] (4+.5,0-2) circle [radius=.04];
			\draw [red] (4+0,.5-2) circle [radius=.04];
			\draw [red] (4+1,.5-2) circle [radius=.04];
			\draw (6+1,1-2)--(6,-2);
			\draw (6,1-2)--(6+1,-2);    
			\draw [fill=white] (6+0,0-2)rectangle (6+1,1-2);
			\draw (6+0,1-2)--(6+.5,.75-2)--(6+1,1-2);
			\draw (6+0,1-2)--(6+.25,.55-2)--(6+0,0-2);
			\draw (6+0,0-2)--(6+.5,.25-2)--(6+1,0-2);
			\draw (6+1,0-2)--(6+.75,.5-2)--(6+1,1-2);
			\draw (6+.25,.5-2)--(6+.5,.25-2)--(6+.75,.5-2)--(6+.5,.75-2)--(6+.25,.5-2);
			\draw [fill=black](6+.25,.5-2) circle [radius=.04];
			\draw [fill=black](6+.75,.5-2) circle [radius=.04];
			\draw [fill=black](6+.5,.75-2) circle [radius=.04];
			\draw [fill=black](6+.5,.25-2) circle [radius=.04];
			\draw [fill=black] (6+0,0-2) circle [radius=.04];
			\draw [fill=black] (6+1,0-2) circle [radius=.04];
			\draw [fill=black] (6+0,1-2) circle [radius=.04];
			\draw [fill=black] (6+1,1-2) circle [radius=.04]; 
			\draw [red] (6+.5,1-2) circle [radius=.04];
			\draw [red] (6+.5,0-2) circle [radius=.04];
			\draw [red] (6+0,.5-2) circle [radius=.04];
			\draw [red] (6+1,.5-2) circle [radius=.04];
			\draw [red] (6+.375,.625-2) circle [radius=.04];
			\draw [red] (6+.625,.625-2) circle [radius=.04];
			\draw [red] (6+.375,.375-2) circle [radius=.04];
			\draw [red] (6+.625,.375-2) circle [radius=.04];
			\draw [red] (6+.5,1-2) circle [radius=.04];
			\draw [red] (6+.5,0-2) circle [radius=.04];
			\draw [red] (6+0,.5-2) circle [radius=.04];
			\draw [red] (6+1,.5-2) circle [radius=.04];
			\draw [red] (6+.25,.875-2) circle [radius=.04];
			\draw [red] (6+.75,.875-2) circle [radius=.04];
			\draw [red] (6+.25,.125-2) circle [radius=.04];
			\draw [red] (6+.75,.125-2) circle [radius=.04];
			\draw [red] (6+.125,.75-2) circle [radius=.04];
			\draw [red] (6+.875,.25-2) circle [radius=.04];
			\draw [red] (6+.125,.25-2) circle [radius=.04];
			\draw [red] (6+.875,.75-2) circle [radius=.04];
			\end{tikzpicture}.
		\end{figure}
		In order to show the above diagrams form a basis of the subspace $\mathscr{Q}$, we need to compute the inner product matrix of the above diagrams and the basis $\mathcal{B}$. It follows from a direct computation that the inner product matrix $M$ is given as
		
		$\begin{bmatrix}
		1&1&1&1&1&1&1&1&1\\
		1&1&1&0&0&0&0&0&0\\
		1&0&0&1&0&0&1&0&0\\
		0&0&0&0&0&0&1&1&1\\
		0&0&1&0&0&1&0&0&1\\
		\binom{n-4}{2}&\binom{n-5}{2}&\binom{n-6}{2}&\binom{n-5}{2}&\binom{n-6}{2}&\binom{n-7}{2}&\binom{n-6}{2}&\binom{n-7}{2}&\binom{n-8}{2}\\
		x_1&x_2&x_3&x_4&x_5&x_6&x_7&x_8&x_9\\
		x_1&x_4&x_7&x_2&x_5&x_8&x_3&x_6&x_9\\
		y_1&y_2&y_3&y_4&y_5&y_6&y_7&y_8&y_9
		\end{bmatrix}$.
		
		\noindent
		In the inner product matrix $M$, $\{x_i:1\leq i\leq9\}$ and $\{y_i:1\leq i\leq9\}$ are given by
		\begin{align*}
		x_1&=\binom{n - 4}{2}\binom{n - 6}{2},\\
		x_2&=\left((n - 5) + \binom{n - 6}{2}\right)\binom{n - 5}{2},\\
		x_3&=\binom{n - 4}{2} +2(n - 6)\binom{n - 5}{2} + 
		\binom{n - 6}{2}\binom{n - 6}{2},\\
		x_4&=\binom{n - 5}{2}\binom{n - 7}{2},\\
		x_5&=(n - 6)\binom{n - 6}{2} + 
		\binom{n - 6}{2}\binom{n - 7}{2},\\
		x_6&=\binom{n - 5}{2} +2(n - 7)\binom{n - 6}{2} + 
		\binom{n - 7}{2}\binom{n - 7}{2},\\
		x_7&=\binom{n - 6}{2}\binom{n - 8}{2},\\
		x_8&=(n - 7)\binom{n - 7}{2} + 
		\binom{n - 7}{2}\binom{n - 8}{2},\\
		x_9&=\binom{n - 6}{2} +2(n - 8)\binom{n - 7}{2} + 
		\binom{n - 8}{2}\binom{n - 8}{2},\\
		y_1&=\frac{1}{16} (2919840 - 3704488 n + 2039584 n^2 - 637336 n^3 + 123793 n^4 - 15324 n^5 + 1182 n^6 - 52 n^7 + n^8),\\
		y_2=y_4&=\frac{1}{16}(3699360 - 4400712 n + 2297408 n^2 - 688048 n^3 + 129385 n^4 - 15652 n^5 + 1190 n^6 - 52 n^7 + n^8),\\
		y_3=y_7&=\frac{1}{16}(4579680 - 5154008 n + 2567320 n^2 - 739896 n^3 + 135017 n^4 - 15980 n^5 + 1198 n^6 - 52 n^7 + n^8),\\
		y_5&=\frac{1}{16}(4564896 - 5145064 n + 2565280 n^2 - 739688 n^3 + 135009 n^4 -15980 n^5 + 1198 n^6 - 52 n^7 + n^8),\\
		y_6=y_8&=\frac{1}{16}(5534816 - 5947448 n + 2845304 n^2 - 792464 n^3 + 140673 n^4 - 16308 n^5 + 1206 n^6 - 52 n^7 + n^8),\\
		y_9&=\frac{1}{16}(6600576 - 6800712 n + 3135568 n^2 - 846168 n^3 + 146369 n^4 - 16636 n^5 + 1214 n^6 - 52 n^7 + n^8).
		\end{align*}
		The determinant of $M$ is 
		\begin{equation*}
		\det{M}=8(2n^2-26n+83).
		\end{equation*}
		It follows that $\det{M}\neq0$ when $n\geq 8$. Therefore, the set $\mathscr{B}$ forms a basis for $\mathscr{Q}$, namely, $\mathscr{Q}\subset\mathscr{A}_4$. Hence, we have that
		\begin{equation}
		\begin{tikzpicture}
		\draw (0,0)--(1,1);
		\draw (0,1)--(1,0);
		\draw (.25,.75)--(.75,.75);
		\draw [red] (.5,.75) circle [radius=.04];
		\draw [fill=black](.25,.75) circle [radius=.04];
		\draw [fill=black](.75,.75) circle [radius=.04];
		\node at (1.75,.5) {$\in\mathscr{A}_{4,+}$};
		\draw [fill=black] (0,0) circle[radius=.04];
		\draw [fill=black] (1,1) circle[radius=.04];
		\draw [fill=black] (1,0) circle[radius=.04];
		\draw [fill=black] (0,1) circle[radius=.04];
		\end{tikzpicture}.
		\end{equation}. 
		\item \textbf{Reduced case:} Now we proceed to the case when $n=7$. In this case, the dimension of $\mathscr{Q}$ reduces to $8$ and the basis $\mathcal{B}$ reduces to 
		\begin{align*}
		b_1&=[\{1,2\},\{3,4\},\{1,2\},\{3,4\}],\\
		b_2&=[\{1,2\},\{3,4\},\{1,2\},\{3,5\}],\\
		b_3&=[\{1,2\},\{3,4\},\{1,2\},\{5,6\}],\\
		b_4&=[\{1,2\},\{3,4\},\{1,5\},\{3,4\}],\\
		b_5&=[\{1,2\},\{3,4\},\{1,5\},\{3,6\}],\\
		b_6&=[\{1,2\},\{3,4\},\{1,5\},\{6,7\}],\\
		b_7&=[\{1,2\},\{3,4\},\{5,6\},\{3,4\}],\\
		b_8&=[\{1,2\},\{3,4\},\{5,6\},\{3,7\}].
		\end{align*}
		It follows by direct computation that 
		\begin{equation}
		M=
		\left[
		\begin{array}{cccccccc}
		1 & 1 & 1 & 1 & 1 & 1 & 1 & 1 \\
		1 & 1 & 1 & 0 & 0 & 0 & 0 & 0 \\
		1 & 0 & 0 & 1 & 0 & 0 & 1 & 0 \\
		0 & 0 & 0 & 0 & 0 & 0 & 1 & 1 \\
		0 & 0 & 1 & 0 & 0 & 1 & 0 & 0 \\
		3 & 1 & 0 & 1 & 0 & 0 & 0 & 0 \\
		0 & 2 & 5 & 0 & 0 & 1 & 0 & 0 \\
		0 & 0 & 0 & 2 & 0 & 0 & 5 & 1 \\
		\end{array}
		\right].
		\end{equation}
		Moreover, $\det M=4\neq 0$. 
		Now we proceed to the case when $n=6$. In this case, the dimension of $\mathscr{Q}$ reduces to $6$ and the basis $\mathcal{B}$ reduces to 
		\begin{align*}
		b_1&=[\{1,2\},\{3,4\},\{1,2\},\{3,4\}],\\
		b_2&=[\{1,2\},\{3,4\},\{1,2\},\{3,5\}],\\
		b_3&=[\{1,2\},\{3,4\},\{1,2\},\{5,6\}],\\
		b_4&=[\{1,2\},\{3,4\},\{1,5\},\{3,4\}],\\
		b_5&=[\{1,2\},\{3,4\},\{1,5\},\{3,6\}],\\
		b_7&=[\{1,2\},\{3,4\},\{5,6\},\{3,4\}].
		\end{align*}
		It follows by direct computation that 
		\begin{equation}
		M=\left[
		\begin{array}{cccccc}
		1 & 1 & 1 & 1 & 1 & 1 \\
		1 & 1 & 1 & 0 & 0 & 0 \\
		1 & 0 & 0 & 1 & 0 & 1 \\
		0 & 0 & 0 & 0 & 0 & 1 \\
		0 & 0 & 1 & 0 & 0 & 0 \\
		3 & 1 & 0 & 1 & 0 & 0 \\
		\end{array}
		\right]
		\end{equation}
		Moreover, $\det M=1\neq 0$. 
	\end{enumerate}
	Therefore, in both cases, we show that $\mathscr{B}$ is a basis for the subspace $Q$. Since $R_S\in Q$ and $\mathscr{B}\subset\mathscr{A}_4$, we have that $R_S\in\mathscr{A}_\bullet$. By Lemma \ref{pro: reduce lemma} and Theorem \ref{lem: symmetric}, we prove the theorem. 
\end{proof}

\section{Appendix}\label{sec: appendix}
Theorem \ref{lem: symmetric} says that in a submodel $\mathscr{A}$ of a spin model, the existence of the transposition $R$ guarantees that the submodel $\mathscr{A}$ carries the symmetry of a finite group. This theorem is first proved independently by Jones\cite{JonPC} and Curtin\cite{Cur03}. In \S \ref{sec: generating property}, we give an alternative proof for the spin model associated to $KG_{n,2}$. In this section, we enhance the statement of Jones and Curtin for general planar algebras.

In \cite{Ren19Uni}, the transposition $R$ and the $\GHZ$ tensor are characterized by the following relations:
\begin{proposition}\label{pro: skein relation for R and GHZ}
	Let $\mathscr{P}_\bullet$ be a spin model. Then the transposition $R\in\mathscr{P}_4$ and $\GHZ\in\mathscr{P}_3$ satisfies the following relations:
	\begin{enumerate}
		\item\label{itm: rel1} Reidemeister moves:
		
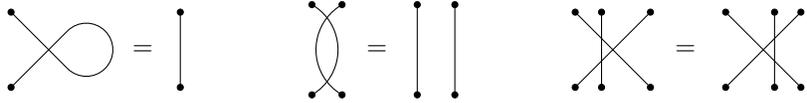
\begin{figure}[H]
			\begin{tikzpicture}
			\begin{scope}
			\draw (-.5,-.5)--(.25,.25);
			\draw (-.5,.5)--(.25,-.25) arc[radius=.25*1.414, start angle=-135,end angle=135];
			\node at (1.25,0) {$=$};
			\draw (1.75,.5)--(1.75,-.5);
			\draw [fill=black] (-.5,.5) circle [radius=.04];
			\draw [fill=black] (-.5,-.5) circle [radius=.04];
			\draw [fill=black] (1.75,.5) circle [radius=.04];
			\draw [fill=black] (1.75,-.5) circle [radius=.04];
			\end{scope}
			\begin{scope}[shift={(3.5,-.6)}]
			\draw (0,1.2) arc [radius=1.2/1.732, start angle=60, end angle=-60];
			\draw (.4,1.2) arc [radius=1.2/1.732, start angle=120, end angle=240];
			\node at (.6,.6) [right] {$=$};
			\draw (1.4,0)--(1.4,1.2);
			\draw (1.9,0)--(1.9,1.2);
			\draw [fill=black] (0,1.2) circle [radius=.04];
			\draw [fill=black] (.4,1.2) circle [radius=.04];
			\draw [fill=black] (0,0) circle [radius=.04];
			\draw [fill=black] (.4,0) circle [radius=.04];
			\draw [fill=black] (1.4,1.2) circle [radius=.04];
			\draw [fill=black] (1.4,0) circle [radius=.04];
			\draw [fill=black] (1.9,1.2) circle [radius=.04];
			\draw [fill=black] (1.9,0) circle [radius=.04];
			\end{scope}
			\begin{scope}[shift={(7,-.5)}]
			\draw (0,0)--(1,1);
			\draw (0,1)--(1,0);
			\draw (.35,0)--(.35,1);
			\node at (1.2,.5) [right] {$=$};
			\draw (2,0)--(3,1);
			\draw (2,1)--(3,0);
			\draw (2.65,0)--(2.65,1);
			\draw [fill=black] (0,1) circle [radius=.04];
			\draw [fill=black] (1,1) circle [radius=.04];
			\draw [fill=black] (.35,1) circle [radius=.04];
			\draw [fill=black] (2,1) circle [radius=.04];
			\draw [fill=black] (3,1) circle [radius=.04];
			\draw [fill=black] (2.65,1) circle [radius=.04];
			\draw [fill=black] (0,0) circle [radius=.04];
			\draw [fill=black] (1,0) circle [radius=.04];
			\draw [fill=black] (.35,0) circle [radius=.04];
			\draw [fill=black] (2,0) circle [radius=.04];
			\draw [fill=black] (3,0) circle [radius=.04];
			\draw [fill=black] (2.65,0) circle [radius=.04];
			\end{scope}
			\end{tikzpicture}.
		\end{figure}
		\item\label{itm: rel2} Flatness: For any $x\in\mathscr{P}_\bullet$, we have that
		
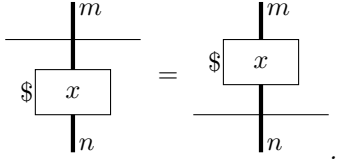
\begin{figure}[H]
			\begin{tikzpicture}
			\draw [ultra thick](.5,1.5)--(.5,-.5);
			\node at (.45,1.4) [right] {$m$};
			\node at (.45,-.4) [right] {$n$};
			\draw (-.4,1)--(1.4,1);
			\draw [fill=white] (0,0) rectangle (1,.6);
			\node at (.5,.3) {$x$};
			\node at (.1,.3) [left] {$\$$};
			\node at (1.5,.5) [right]{$=$};
			\draw [ultra thick](3,-.5)--(3,1.5);
			\draw [fill=white] (2.5,1) rectangle (3.5,.4);
			\node at (2.95,1.4)[right] {$m$};
			\node at (2.95,-.4) [right] {$n$};
			\draw(2.1,0)--(3.9,0);
			\node at (3,.7) {$x$};
			\node at (2.6,.7) [left] {$\$$};
			\end{tikzpicture}.
		\end{figure}
		\item\label{itm: rel3} Frobenius relations:
		
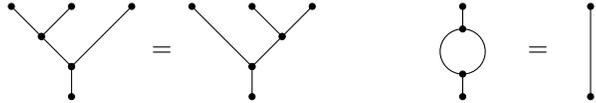
\begin{figure}[H]
			\begin{tikzpicture}
			\begin{scope}[scale=.8]
			\draw  (0,0)--(.5,-.5)--(1,0);
			\draw  (.5,-.5)--(1,-1)--(2,0);
			\draw  (1,-1)--(1,-1.5);
			\node at (2.5,-.75) {$=$};
			\draw (3,0)--(4,-1)--(5,0);
			\draw  (4,0)--(4.5,-.5);
			\draw  (4,-1.5)--(4,-1);
			\draw [fill=black] (.5,-.5) circle [radius=.05];
			\draw [fill=black] (1,-1) circle [radius=.05];
			\draw [fill=black] (4.5,-.5) circle [radius=.05];
			\draw [fill=black] (4,-1) circle [radius=.05];
			\draw [fill=black] (.5,-.5) circle [radius=.05];
			\draw [fill=black] (1,-1) circle [radius=.05];
			\draw [fill=black] (4.5,-.5) circle [radius=.05];
			\draw [fill=black] (4,-1) circle [radius=.05];
			\draw [fill=black] (0,0) circle [radius=.05];
			\draw [fill=black] (2,0) circle [radius=.05];
			\draw [fill=black] (1,0) circle [radius=.05];
			\draw [fill=black] (1,-1.5) circle [radius=.05];
			\draw [fill=black] (3,0) circle [radius=.05];
			\draw [fill=black] (4,0) circle [radius=.05];
			\draw [fill=black] (5,0) circle [radius=.05];            
			\draw [fill=black] (4,-1.5) circle [radius=.05];
			\end{scope}
			\begin{scope}[shift={(6,-.6)}]
			\draw  (0,-.6)--(0,.6);
			\draw [fill=white](0,0) circle [radius=.3];
			\node at (1,0) {$=$};
			\draw (1.7,-.6)--(1.7,.6);
			\draw [fill=black] (0,.3) circle [radius=.04];
			\draw [fill=black] (0,-.3) circle [radius=.04];
			\draw [fill=black] (0,.6) circle [radius=.04];
			\draw [fill=black] (0,-.6) circle [radius=.04];
			\draw [fill=black] (1.7,.6) circle [radius=.04];
			\draw [fill=black] (1.7,-.6) circle [radius=.04];
			\end{scope}
			\end{tikzpicture}.
		\end{figure}
	\end{enumerate}    
\end{proposition}
Let $\mathscr{Q}_\bullet$ be a positive planar algebra such that there exists $S\in\mathscr{Q}_4$ and $W\in\mathscr{Q}_3$ satisfying relations \eqref{itm: rel1}, \eqref{itm: rel2} and \eqref{itm: rel3}. Note that relations \eqref{itm: rel1},\eqref{itm: rel2} and \eqref{itm: rel3} provide the skein theory for partition planar algebras \cite{Jon94}. This leads to the following corollaries:
\begin{enumerate}
	\item\label{itm: cor1Jones} The circle parameter for $\mathscr{Q}_\bullet$ is an integer, namely, there exists $d\in\mathbb{N}$ such that
	\begin{figure}[H]
		\begin{tikzpicture}
		\draw (0,0) circle [radius=.4];
		\node at (.75,0) {$=d$};
		\end{tikzpicture}.
	\end{figure}
	\item\label{itm: cor2Jones} Let $\mathscr{P}_\bullet^{S_d}$ be the group-action model associated to $S_d\curvearrowright \{1,2,\cdots,d\}$. Then there exists planar algebraic homomorphism $\alpha$ from $\mathscr{P}_\bullet^{S_d}$ to $\mathscr{Q}_\bullet$ such that $\alpha(R)=S$ and $\alpha(\GHZ)=W$.
\end{enumerate}
\begin{proposition}
	For every $n\in\mathbb{N}$, there exists a homomorphism from $S_n$ to $\mathscr{Q}_2n$ by sending the permutation $(k,k+1)$ to 
	\begin{figure}[H]
		\begin{tikzpicture}
		\draw (-.5,-.5)--(.5,.5);
		\draw (-.5,.5)--(.5,-.5);
		\draw [fill=black] (-.5,-.5) circle [radius=.04];
		\draw [fill=black] (.5,-.5) circle [radius=.04];
		\draw [fill=black] (-.5,.5) circle [radius=.04];
		\draw [fill=black] (.5,.5) circle [radius=.04];
		\draw [ultra thick] (-1.2,.55)--(-1.2,-.55);
		\node at (-1.3,0) [left] {$k-1$};
		\draw [ultra thick] (1.2,.55)--(1.2,-.55);
		\node at (1.3,0) [right] {$n-k-1$};
		\end{tikzpicture}.
	\end{figure}
\end{proposition}
\begin{proof}
	It follows that $S$ is a symmetric braiding from Relation \eqref{itm: rel1}. 
\end{proof}
\begin{remark}
	Suppose $\sigma\in S_n$. We represent it by the diagram $\raisebox{-4mm}{\begin{tikzpicture}
		\draw [ultra thick] (.5,-.4)--(.5,.9);
		\draw [fill=white] (.75,0) rectangle (.25,.5);
		\node at (.5,.25) {$\sigma$};
		\node at (.33,.25) [left] {$\$$};
		\node at (.35,.7) {$m$};
		\node at (.35,-.2) {$m$};
		\end{tikzpicture}}$. 
\end{remark}
\begin{proposition}
	Suppose $n\in\mathbb{N}$, we define the following operations on $\mathscr{Q}_n:$
	\begin{enumerate}
		\item A binary operation $\circ:$ given $f,g\in \mathscr{Q}_n$,
		\begin{figure}[H]
			\begin{tikzpicture}
			\draw [ultra thick] (.5,0)--(1.25,.5)--(2,0);
			\draw [ultra thick] (1.25,.5)--(1.25,1);
			\draw [fill=white] (1.25,.5) circle [radius=.2];
			\node at (1.25,.5) {$n$};
			\node at (-.5,-.3)[left]  {$f\circ g=$};
			\draw [fill=white] (0,0) rectangle (1,-.6);
			\draw [fill=white](1.5,0) rectangle (2.5,-.6);
			\node at (-.1,-.3) {$\$$};
			\node at (.5,-.3) {$f$};
			\node at (2,-.3) {$g$};
			\node at (1.4,-.3) {$\$$};
			\end{tikzpicture}.
		\end{figure}
		\item The involution $^\dagger:$ given $f\in\mathscr{Q}_n$, let $\sigma$ be the permutation on $\{1,2,\cdots,n\}$ such that $\sigma(j)=n-j$ for $1\leq j\leq n$,
		\begin{figure}[H]
			\begin{tikzpicture}
			\draw [fill=white] (0,0) rectangle (1,.6);
			\draw [ultra thick] (.5,0) arc[radius=.75, start angle=180, end angle=360]--(2,1.1);
			\draw [fill=white] (1.5,0) rectangle (2.5,.6);
			\node at (.5,.3) {$f^*$};
			\node at (2,.3) {$\sigma$};
			\node at (-.1,.3) {$\$$};
			\node at (1.4,.3) {$\$$};
			\node at (-.2,.3) [left] {$f^\dagger=$};
			\end{tikzpicture}.
		\end{figure}        
		\item The norm $||\cdot||:$ we equip $\mathscr{Q}_n$ with the inner product $\langle\cdot,\cdot\rangle:$ for $x,y\in\mathscr{Q}_n$,
		\begin{figure}[H]
			\begin{tikzpicture}
			\draw (0,0) rectangle (1,.6);
			\draw (0,1.1) rectangle (1,1.7);
			\draw [ultra thick] (.5,.6)--(.5,1.1);
			\node at (.35,.85) {$n$};
			\node at (.5,.3) {$x$};
			\node at (.5,1.4) {$y^*$};
			\node at (-.1,.3) {$\$$};
			\node at (-.1,1.4) {$\$$};
			\end{tikzpicture}.
		\end{figure} 
		We denote the norm on $\mathscr{Q}_n$ by $\langle\cdot,\cdot\rangle$ by $||\cdot||_2$. Suppose $f\in\mathscr{Q}_n$, we define the norm $||\cdot||$ by
		\begin{equation}
		||f||=\sup\{||f\circ x||_2:||x||_2=1\}.
		\end{equation}
	\end{enumerate}
	Then $(\mathscr{P}_n,\circ,^\dagger,||\cdot||)$ is a commutative $C^*$-algebra and we call it the Hadamard algebra.  
\end{proposition}
\begin{proof}
	This proposition follows directly from the skein relations of $S$ and $W$. 
\end{proof}

\begin{lemma}\label{lem: minimality}
	Suppose $p$ is a minimal projection in $\mathscr{Q}_m$ for some $m\in\mathbb{N}$. Then for any $g\in S_m$, the following are minimal projections:
	
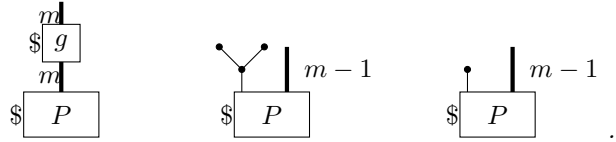
\begin{figure}[H]
		\begin{tikzpicture}
		\begin{scope}[shift={(-2.8,0)}]
		\draw (0,0) rectangle (1,.6);
		\node at (.5,.3) {$P$};
		\node at (-.1,.3) {$\$$};
		\draw [ultra thick](.5,.6)--(.5,1.8);
		\draw [fill=white] (.25,1) rectangle (.75,1.5);
		\node at (.5,1.25) {$g$};
		\node at (.15,1.25) {$\$$};
		\node at (.35,.8) {$m$};
		\node at (.35,1.6) {$m$};
		\end{scope}
		\begin{scope}
		\draw (0,0) rectangle (1,.6);
		\node at (.5,.3) {$P$};
		\node at (-.1,.3) {$\$$};
		\draw (.1,.6)--(.1,.9);
		\draw (-.2,1.2)--(.1,.9)--(.4,1.2);
		\draw [ultra thick] (.7,.6)--(.7,1.2);
		\draw [fill=black] (.1,.9) circle [radius=.04];
		\draw [fill=black] (-.2,1.2) circle [radius=.04];
		\draw [fill=black] (.4,1.2) circle [radius=.04];
		\node at (.8,.9) [right]{$m-1$};
		\end{scope}
		\begin{scope}[shift={(3,0)}]
		\draw (0,0) rectangle (1,.6);
		\node at (.5,.3) {$P$};
		\node at (-.1,.3) {$\$$};
		\draw (.1,.6)--(.1,.9);
		\draw [fill=black] (.1,.9) circle [radius=.04];
		\draw [ultra thick] (.7,.6)--(.7,1.2);
		\node at (.8,.9) [right]{$m-1$};
		\end{scope}
		\end{tikzpicture}.
	\end{figure}
\end{lemma}
\begin{proof}
	This lemma follows directly from the skein relations of $S$ and $W$. We will only show the first one as an example. Let $x\in\mathscr{Q}_m$. Then we have that
	\begin{figure}[H]
		\begin{tikzpicture}
		\begin{scope}
		\node at (-.3,.3) [left] {$(g\cdot p)\circ x=$};
		\draw (0,0) rectangle (1,.6);
		\node at (.5,.3) {$P$};
		\node at (-.1,.3) {$\$$};
		\draw [ultra thick](.5,.6)--(.5,1.5) arc[radius=.1, start angle=180, end angle=90]--(1.9,1.6) arc[radius=.1, start angle=90, end angle=0]--(2,.6);
		\draw [fill=white] (.25,1) rectangle (.75,1.5);
		\node at (.5,1.25) {$g$};
		\node at (.15,1.25) {$\$$};
		\node at (.35,.8) {$m$};
		\draw (1.5,0) rectangle (2.5,.6);
		\node at (2,.3) {$x$};
		\node at (1.4,.3) {$\$$};
		\draw [ultra thick] (1.25,1.6)--(1.25,2.6);
		\draw [fill=white] (1.25,1.6) circle [radius=.2];
		\node at (1.25,1.6) {$m$};
		\end{scope}
		\begin{scope}[shift={(3.5,0)}]
		\node at (-.5,.3) {$=$};
		\draw (0,0) rectangle (1,.6);
		\node at (.5,.3) {$P$};
		\node at (-.1,.3) {$\$$};
		\draw [ultra thick](.5,.6)--(.5,1.5) arc[radius=.1, start angle=180, end angle=90]--(1.9,1.6) arc[radius=.1, start angle=90, end angle=0]--(2,.6);
		\draw [fill=white] (1.75,1) rectangle (2.25,1.5);
		\node at (2,1.25) {$g$};
		\node at (2.35,1.25) {$\$$};
		\node at (2.25,.8) {$m$};
		\draw (1.5,0) rectangle (2.5,.6);
		\node at (2,.3) {$x$};
		\node at (1.4,.3) {$\$$};
		\draw [ultra thick] (1.25,1.6)--(1.25,2.6);
		\draw [fill=white] (1.25,1.6) circle [radius=.2];
		\node at (1.25,1.6) {$m$};
		\draw [fill=white] (1,1.9) rectangle (1.5,2.4);
		\node at (1.25,2.15) {$g$};
		\node at (.9,2.15) {$\$$};
		\end{scope}
		\node at (6.2,.3) [right] {$=g\cdot (p\circ (g^{-1}\cdot x))$.};
		\end{tikzpicture}
	\end{figure}
	Since $p$ is a minimal projection, we know the the right hand side is nonzero if and only if $p=g^{-1}\cdot x$ if and only if $g\cdot p=x$. This implies that $g\cdot p$ is a minimal projection. 
\end{proof}

\begin{theorem}\label{thm: key}
	Let $\mathscr{Q}_\bullet$ be a planar algebra such that there exists $R\in\mathscr{Q}_4$ and $W\in\mathscr{Q}_3$ satisfying relations \eqref{itm: rel1}, \eqref{itm: rel2} and \eqref{itm: rel3}. Then there exists $d\in\mathbb{N}$ and a finite group $G\leq S_d$ such that $\mathscr{Q}_\bullet$ is isomorphic to $\mathscr{P}_\bullet^{G}$.     
\end{theorem}
\begin{proof}
	Let $d$ be the circle parameter of $\mathscr{Q}_\bullet$. By Corollary \eqref{itm: cor1Jones},we have that $d$ is an integer. By universal skein theory for group actions \cite{Ren19Uni}, the group-action model $\mathscr{P}_\bullet^{S_d}$ is generated by $\GHZ$, the transposition $R$ and the molecule $Y=\chi_{S_d\cdot(1,2,\cdots,d)}$. By Universal skein theory for group-actions \cite{Ren19Uni}, we know that
	\begin{align}
	&\begin{tikzpicture}
	\begin{scope}
	\draw (0,0) rectangle (1,.6);
	\draw (1.5,0) rectangle (2.5,.6);
	\draw [ultra thick] (.5,.6)--(.5,1.5);
	\draw [ultra thick] (2,.6)--(2,1.5);
	\node at (-.1,.3) {$\$$};
	\node at (1.4,.3) {$\$$};
	\node at (.5,.3) {$Y$};
	\node at (2,.3) {$Y$};
	\node at (3,.3) {$=$};
	\node at (.35,1) {$d$};
	\node at (1.85,1) {$d$};
	\end{scope}
	\begin{scope}[shift={(4.2,0)}]
	\node at (-.5,.2) {$\displaystyle\sum_{g\in S_d}$};
	\draw (0,0) rectangle (1,.6);
	\draw [ultra thick] (.5,.6)--(.5,.8);
	\draw [ultra thick] (.5,1)--(.8,1.3) arc[radius=.2*1.414/.414, start angle=-45, end angle=0]--(1,2.3);
	\draw [ultra thick] (.5,1)--(.2,1.3) arc[radius=.2*1.414/.414, start angle=225, end angle=180]--(0,2.3);
	\draw [fill=white] (.75,2) rectangle (1.25,1.5);
	\draw [fill=white] (.5,1) circle [radius=.2];
	\node at (1,1.75) {$g$};
	\node at (.65,1.75) {$\$$};
	\node at (.5,.3) {$Y$};
	\node at (.5,1) {$d$};
	\end{scope}
	\end{tikzpicture}\label{equ: HighforY},\\
	&\begin{tikzpicture}
	\node at (-.4,.85) [left]{$Y^*Y=$};
	\draw (0,0) rectangle (1,.6);
	\draw (0,1.1) rectangle (1,1.7);
	\draw [ultra thick] (.5,.6)--(.5,1.1);
	\node at (.35,.85) {$d$};
	\node at (.5,.3) {$Y$};
	\node at (.5,1.4) {$Y^*$};
	\node at (-.1,.3) {$\$$};
	\node at (-.1,1.4) {$\$$};
	\node at (2,.85) {$=|S_d|$};
	\end{tikzpicture}.
	\end{align}
	It follows that the image of $Y$ under $\alpha$ is a projection in the Hadamard algebra $(\mathscr{Q}_d,\circ)$. We still denote it by $Y$ with abusing of notations. Since $(\mathscr{Q}_d,\circ)$ is commutative, there exist orthogonal minimal projections $Y_1,Y_2,\cdots,Y_m$ such that 
	\begin{equation}
	Y=Y_1+Y_2+\cdots+Y_m. 
	\end{equation}
	Let $X=\{Y_1,Y_2,\cdots,Y_m\}$. For every $g\in S_d$ and $1\leq j\leq m$, we have that $g\cdot Y_j$ is also a minimal subprojection of $Y$ by Lemma \ref{lem: minimality}. Therefore, there exists an action of $S_d$ on $X$ which permutes these minimal projections. For every $1\leq j\leq m$, let $G_i$ be the stabilizer of $Y_i$, namely, 
	\begin{equation}
	G_j=\{g\in S_d: g\cdot Y_j=Y_j\}.
	\end{equation}
	By Equation \eqref{equ: HighforY}, we know that
	\begin{equation}\label{equ: higer for Yi}
	\begin{tikzpicture}
	\begin{scope}[scale=.75]
	\begin{scope}
	\draw (0,0) rectangle (1,.6);
	\draw (1.5,0) rectangle (2.5,.6);
	\draw [ultra thick] (.5,.6)--(.5,1.5);
	\draw [ultra thick] (2,.6)--(2,1.5);
	\node at (-.1,.3) {$\$$};
	\node at (1.4,.3) {$\$$};
	\node at (.5,.3) {$Y_i$};
	\node at (2,.3) {$Y_i$};
	\node at (3,.3) {$=$};
	\node at (.35,1) {$d$};
	\node at (1.85,1) {$d$};
	\end{scope}
	\begin{scope}[shift={(3.5,0)}]
	\draw (0,0) rectangle (1,.6);
	\draw (1.5,0) rectangle (2.5,.6);
	\draw [ultra thick] (.5,.6)--(.5,1.5);
	\draw [ultra thick] (2,.6)--(2,1.5);
	\node at (-.1,.3) {$\$$};
	\node at (1.4,.3) {$\$$};
	\node at (.5,.3) {$Y_i$};
	\node at (2,.3) {$Y_i$};
	\begin{scope}[shift={(3,0)}]
	\draw (0,0) rectangle (1,.6);
	\draw (1.5,0) rectangle (2.5,.6);
	\draw [ultra thick] (.5,.6)--(.5,1.5);
	\draw [ultra thick] (2,.6)--(2,1.5);
	\node at (-.1,.3) {$\$$};
	\node at (1.4,.3) {$\$$};
	\node at (.5,.3) {$Y$};
	\node at (2,.3) {$Y$};
	\end{scope}
	\draw [ultra thick](.5,1.5) arc[radius=.3,start angle=180, end angle=90]--(3.2,1.8) arc[radius=.3, start angle=90, end angle=0];
	\draw [ultra thick] (2,1.5)--(2,2) arc[radius=.3,start angle=180, end angle=90]--(4.7,2.3) arc[radius=.3, start angle=90, end angle=0]--(5,1.5);
	\draw [ultra thick] (1.5,1.8)--(1.5,2.8);
	\draw [ultra thick] (3,2.3)--(3,2.8);
	\draw [fill=white] (1.5,1.8) circle [radius=.2];
	\node at (1.5,1.8) {$d$};
	\draw [fill=white] (3,2.3) circle [radius=.2];
	\node at (3,2.3) {$d$};
	\end{scope}
	\begin{scope}[shift={(11,0)}]
	\draw (0,0) rectangle (1,.6);
	\draw (1.5,0) rectangle (2.5,.6);
	\draw [ultra thick] (.5,.6)--(.5,1.5);
	\draw [ultra thick] (2,.6)--(2,1.5);
	\node at (-.1,.3) {$\$$};
	\node at (1.4,.3) {$\$$};
	\node at (.5,.3) {$Y_i$};
	\node at (2,.3) {$Y_i$};
	\begin{scope}[shift={(3,0)}]
	\node at (-.1,.3) {$\$$};
	\draw (0,0) rectangle (1,.6);
	\draw [ultra thick] (.5,.6)--(.5,.8);
	\draw [ultra thick] (.5,1)--(.8,1.3) arc[radius=.2*1.414/.414, start angle=-45, end angle=0]--(1,2.3);
	\draw [ultra thick] (.5,1)--(.2,1.3) arc[radius=.2*1.414/.414, start angle=225, end angle=180]--(0,1.8);
	\draw [fill=white] (.75,2) rectangle (1.25,1.5);
	\draw [fill=white] (.5,1) circle [radius=.2];
	\node at (1,1.75) {$g$};
	\node at (.65,1.75) {$\$$};
	\node at (.5,.3) {$Y$};
	\node at (.5,1) {$d$};
	\end{scope}
	\draw [ultra thick] (.5,1.5)--(.5,1.8) arc[radius=.3, start angle=180, end angle=90]--(2.7,2.1) arc[radius=.3, start angle=90, end angle=0];          
	\draw [ultra thick] (2,1.5)--(2,2.5)arc[radius=.3, start angle=180, end angle=90]--(3.7,2.8) arc[radius=.3, start angle=90, end angle=0]--(4,2);
	\draw [ultra thick] (2.7,2.8)--(2.7,3.3);
	\draw [ultra thick] (1.2,2.1)--(1.2,3.3);
	\draw [fill=white] (2.7,2.8) circle [radius=.2];
	\node at (2.7,2.8) {$d$};
	\draw [fill=white] (1.2,2.1) circle [radius=.2];
	\node at (1.2,2.1) {$d$};        
	\node at (-.3,.1) [left] {$=\displaystyle\sum_{g\in S_d}$};
	\end{scope}
	\begin{scope}[shift={(4.2,-4)}]
	\draw (0,0) rectangle (1,.6);
	\draw (1.5,0) rectangle (2.5,.6);
	\draw [ultra thick] (.5,.6)--(.5,.8);
	\draw [ultra thick] (2,.6)--(2,1.5);
	\node at (-.1,.3) {$\$$};
	\node at (1.4,.3) {$\$$};
	\node at (.5,.3) {$Y_i$};
	\node at (2,.3) {$Y_i$};
	\begin{scope}[shift={(3,0)}]
	\draw (0,0) rectangle (1,.6);
	\draw [ultra thick] (.5,.6)--(.5,.8);
	\node at (-.1,.3) {$\$$};
	\node at (.5,.3) {$Y$};
	\end{scope}
	\draw [ultra thick] (.5,.8) arc[radius=.3, start angle=180, end angle=90]--(3.2,1.1) arc[radius=.3, start angle=90, end angle=0];          
	\draw [ultra thick] (2,1.5)--(2,2.5)arc[radius=.3, start angle=180, end angle=90]--(3.7,2.8) arc[radius=.3, start angle=90, end angle=0]--(4,2);
	\draw [ultra thick] (2.7,2.8)--(2.7,3.3);
	\draw [ultra thick] (1.2,3.3)--(1.2,2) arc[radius=.3, start angle=180, end angle=270]--(3.7,1.7) arc[radius=.3, start angle=270, end angle=360];
	\draw [ultra thick] (1.2,2.1)--(1.2,3.3);
	\draw [ultra thick] (2.7,1.1)--(2.7,1.7);
	\draw [fill=white] (2.7,1.1) circle [radius=.2];
	\node at (2.7,1.1) {$d$};
	\draw [fill=white] (2.7,1.7) circle [radius=.2];
	\node at (2.7,1.7) {$d$};        
	\draw [fill=white] (2.7,2.8) circle [radius=.2];
	\draw [fill=white] (3.75,2.5) rectangle (4.25,2);
	\node at (4,2.25) {$g$};
	\node at (3.65,2.25) {$\$$};      
	\node at (2.7,2.8) {$d$};
	\node at (-.3,.1) [left] {$=\displaystyle\sum_{g\in S_d}$};
	\end{scope}
	\begin{scope}[shift={(9.5,-4)}]
	\draw (1.5,0) rectangle (2.5,.6);
	\draw [ultra thick] (2,.6)--(2,1.5);
	\node at (1.4,.3) {$\$$};
	\node at (2,.3) {$Y_i$};
	\begin{scope}[shift={(3,0)}]
	\draw (0,0) rectangle (1,.6);
	\draw [ultra thick] (.5,.6)--(.5,1.4);
	\node at (-.1,.3) {$\$$};
	\node at (.5,.3) {$Y_i$};
	\end{scope}
	\draw [ultra thick] (2,1.5)--(2,2.5)arc[radius=.3, start angle=180, end angle=90]--(3.7,2.8) arc[radius=.3, start angle=90, end angle=0]--(4,2);
	\draw [ultra thick] (2.7,2.8)--(2.7,3.3);
	\draw [ultra thick] (1.2,3.3)--(1.2,1.7) arc[radius=.3, start angle=180, end angle=270]--(3.7,1.4) arc[radius=.3, start angle=270, end angle=360]--(4,2);
	\draw [ultra thick] (1.2,2.1)--(1.2,3.3);
	\draw [fill=white] (3.5,1.4) circle [radius=.2];
	\node at (3.5,1.4) {$d$};        
	\draw [fill=white] (2.7,2.8) circle [radius=.2];
	\draw [fill=white] (3.75,2.5) rectangle (4.25,2);
	\node at (4,2.25) {$g$};
	\node at (3.65,2.25) {$\$$};      
	\node at (2.7,2.8) {$d$};
	\node at (.9,.1) [left] {$=\displaystyle\sum_{g\in S_d}$};
	\end{scope}
	\begin{scope}[shift={(4,-8)}]
	\draw (0,0) rectangle (1,.6);
	\draw (1.5,0) rectangle (2.5,.6);
	\node at (-.1,.3) {$\$$};
	\node at (1.4,.3) {$\$$};
	\node at (.5,.3) {$Y_i$};
	\node at (2,.3) {$Y_i$};
	\draw [ultra thick] (2,.6)--(2,1.3) arc[radius=.2, start angle=0, end angle=90]--(.7,1.5) arc[radius=.2, start angle=90, end angle=180]--(.5,.6);
	\draw [ultra thick] (1.25,1.5)--(1.25,2.1);
	\draw [ultra thick] (.5,3)--(.5,2.3) arc[radius=.2, start angle=180, end angle=270]--(1.8,2.1) arc[radius=.2, start angle=270, end angle=360]--(2,3);
	\draw [fill=white](1.75,.8) rectangle (2.25,1.3);
	\node at (2,1.05) {$g$};
	\node at (2.35,1.05) {$\$$};
	\draw [fill=white](1.75,2.3) rectangle (2.25,2.8);
	\node at (2,2.55) {$g$};
	\node at (1.65,2.55) {$\$$};
	\node at (0,.1) [left] {$=\displaystyle\sum_{g\in S_d}$};
	\end{scope}
	\begin{scope}[shift={(9,-8)}]
	\draw (0,0) rectangle (1,.6);
	\node at (-.1,.3) {$\$$};
	\node at (.5,.3) {$Y_i$};
	\draw [ultra thick] (.5,.6)--(.5,1);
	\draw [ultra thick] (0,2.5)--(0,1.2) arc[radius=.2, start angle=180, end angle=270]--(.8,1) arc[radius=.2, start angle=270, end angle=360]--(1,2.5);
	\draw [fill=white] (.5,1) circle [radius=.2];
	\node at (.5,1) {$d$};
	\draw [fill=white] (.75,1.5) rectangle (1.25,2);
	\node at (1,1.75) {$g$};
	\node at (.65,1.75) {$\$$};
	\node at (-.1,.1) [left] {$=\displaystyle\sum_{g\in G_i}$};
	\end{scope}
	\end{scope}
	\end{tikzpicture}.
	\end{equation}
	By comparing the $x^*x$ where $x$ is the left- and right-hand side of Equation \eqref{equ: higer for Yi}, we have that
	\begin{align}
	(Y_i^*Y_i)^2&=|G_i| Y_i^*Y_i,\\
	(\Rightarrow)Y_i^*Y_i&=|G_i|.
	\end{align}
	This implies that
	\begin{equation}
	\sum_{i=1}^m |G_i|=\sum_{i=1}^mY_i^*Y_i=Y^*Y=|S_d|.
	\end{equation}
	Let $O_i$ be the orbit of $Y_i$. By the Orbit-Stabilizer Theorem, we have that
	\begin{equation}
	|G_i||O_i|=|S_d|.
	\end{equation}
	By summing over $i$, this implies that
	\begin{equation}
	\sum_{i=1}^m|G_i||O_i|=m|S_d|=m \sum_{i=1}^m |G_i|.
	\end{equation}
	Note that $|O_i|\leq m$ for every $1\leq i\leq m$. This forces that $|O_i|=m$ for every $1\leq i\leq m$. Therefore, for every $2\leq j\leq m$, there exists $g\in S_d$ such that $Y_j=g\cdot Y_1$. Moreover, this implies that all the $G_i$'s are isomorphic and we denote it by $G$.
	
	Now we show that the planar algebra $\mathscr{Q}_\bullet$ is generated by $\{S,W,Y_1\}$. Let $p$ be an arbitrary minimal projection in $\mathscr{Q}_m$ for some $m\in\mathbb{N}$. Note that $\mathscr{P}_m^{S_d}\subset\mathscr{Q}_m$. There exists a projection in $\mathscr{P}_{m}^{S_d}$ such that $p$ is a subprojection of $S$. By Universal skein theory for group-action models \cite{Ren19Uni}, there exist $T\in\mathscr{Q}_{d+m}$ such that 
	\begin{equation}\label{equ: decom of P}
	\begin{tikzpicture}
	\draw (0,0) rectangle (1,.6);
	\node at (.5,.3) {$Y$};
	\node at (-.1,.3) {$\$$};
	\draw [ultra thick](.5,.6)--(.5,1.8);
	\draw [fill=white] (0,1) rectangle (1,1.4);
	\node at (.5,1.2) {$T$};
	\node at (-.1,1.2) {$\$$};
	\node at (.35,.8) {$d$};
	\node at (.35,1.6) {$m$};
	\node at (-.4,.3)[left] {$P=$};
	\begin{scope}[shift={(2.6,0)}]
	\draw (0,0) rectangle (1,.6);
	\node at (.5,.3) {$Y_i$};
	\node at (-.1,.3) {$\$$};
	\draw [ultra thick](.5,.6)--(.5,1.8);
	\draw [fill=white] (0,1) rectangle (1,1.4);
	\node at (.5,1.2) {$T$};
	\node at (-.1,1.2) {$\$$};
	\node at (.35,.8) {$d$};
	\node at (.35,1.6) {$m$};
	\node at (-.3,.3)[left] {$=\displaystyle\sum_{i=1}^m$};
	\end{scope}
	\end{tikzpicture}.
	\end{equation}
	Let $P_i$ be the $i$-th term in the right hand side of Equation \eqref{equ: decom of P} for $1\leq i\leq m$. By Lemma \ref{lem: minimality}, we know that each $P_i$ is a minimal projection. Since $p$ is also a minimal projection, there must exist $1\leq i\leq m$ such that $p=P_i$. Note that for every $1\leq i\leq m$, there exists $g\in S_d$ such that $Y_i=g\cdot Y_1$. This implies that $p$ is generated by $\{S,W,Y_1\}$. Moreover, the generators satisfy the universal skein theory for group-action models. Therefore, we have that $\mathscr{Q}_\bullet$ is isomorphic to $\mathscr{P}_\bullet^G$. 
\end{proof}

\bibliography{bibliography}

\providecommand{\bysame}{\leavevmode\hbox to3em{\hrulefill}\thinspace}
\providecommand{\MR}{\relax\ifhmode\unskip\space\fi MR }
\providecommand{\MRhref}[2]{%
  \href{http://www.ams.org/mathscinet-getitem?mr=#1}{#2}
}
\providecommand{\href}[2]{#2}
\begin{thebibliography}{MRV19}

\bibitem[Ati88]{Ati88}
Michael~F Atiyah, \emph{Topological quantum field theory}, Publications
  Math{\'e}matiques de l'IH{\'E}S \textbf{68} (1988), 175--186.

\bibitem[Ban05]{Ban05}
Teodor Banica, \emph{Quantum automorphism groups of homogeneous graphs},
  Journal of Functional Analysis \textbf{224} (2005), no.~2, 243--280.

\bibitem[BB07]{BanBic07}
Teodor Banica and Julien Bichon, \emph{Quantum automorphism groups of
  vertex-transitive graphs of order $\leq$ 11}, Journal of Algebraic
  Combinatorics \textbf{26} (2007), no.~1, 83.

\bibitem[Bis97]{Bis97}
Dietmar Bisch, \emph{Bimodules, higher relative commutants and the fusion
  algebra associated to a subfactor}, The Fields Institute for Research in
  Mathematical Sciences Communications Series \textbf{13} (1997), 13--63.

\bibitem[BJ00]{BisJon00}
Dietmar Bisch and Vaughan F.~R. Jones, \emph{Singly generated planar algebras
  of small dimension}, Duke Mathematical Journal \textbf{101} (2000), no.~1,
  41--75.

\bibitem[BJ03]{BisJon03}
\bysame, \emph{Singly generated planar algebras of small dimension, part {II}},
  Advances in Mathematics \textbf{175} (2003), no.~2, 297--318.

\bibitem[BJL17]{BJL17}
Dietmar Bisch, Vaughan F.~R. Jones, and Zhengwei Liu, \emph{Singly generated
  planar algebras of small dimension, part {III}}, Transactions of the American
  Mathematical Society \textbf{369} (2017), no.~4, 2461--2476.

\bibitem[BW89]{BirWen89}
Joan~S Birman and Hans Wenzl, \emph{Braids, link polynomials and a new
  algebra}, Transactions of the American Mathematical Society \textbf{313}
  (1989), no.~1, 249--273.

\bibitem[Cha19]{Cha19}
Arthur Chassaniol, \emph{Study of quantum symmetries for vertex-transitive
  graphs using intertwiner spaces}, arXiv preprint arXiv:1904.00455 (2019).

\bibitem[Cur03]{Cur03}
Brian Curtin, \emph{Some planar algebras related to graphs}, Pacific Journal of
  Mathematics \textbf{209} (2003), no.~2, 231--248.

\bibitem[HS68]{HigSims68}
Donald~G Higman and Charles~C Sims, \emph{A simple group of order 44,352,000},
  Mathematische Zeitschrift \textbf{105} (1968), no.~2, 110--113.

\bibitem[Jae92]{Jae92}
Fran{\c{c}}ois Jaeger, \emph{Strongly regular graphs and spin models for the
  {K}auffman polynomial}, Geometriae Dedicata \textbf{44} (1992), no.~1,
  23--52.

\bibitem[Jon]{JonPC}
Vaughan F.~R. Jones, \emph{private communication}.

\bibitem[Jon83]{Jon83}
\bysame, \emph{Index for subfactors}, Inventiones mathematicae \textbf{72}
  (1983), no.~1, 1--25.

\bibitem[Jon94]{Jon94}
\bysame, \emph{The potts model and the symmetric group}, Subfactors (Kyuzeso,
  1993) (1994), 259--267.

\bibitem[Jon99]{Jon99}
\bysame, \emph{Planar algebras, {I}}, arXiv preprint math/9909027 (1999).

\bibitem[Liu15]{LiuYB}
Zhengwei Liu, \emph{Yang-baxter relation planar algebras}, arXiv preprint
  arXiv:1507.06030 (2015).

\bibitem[LMR17]{LMR17}
Martino Lupini, Laura Man{\v{c}}inska, and David~E Roberson, \emph{Nonlocal
  games and quantum permutation groups}, arXiv preprint arXiv:1712.01820
  (2017).

\bibitem[MRV19]{MRV19}
Benjamin Musto, David Reutter, and Dominic Verdon, \emph{The morita theory of
  quantum graph isomorphisms}, Communications in Mathematical Physics
  \textbf{365} (2019), no.~2, 797--845.

\bibitem[Mur90]{Mur90}
Jun Murakami, \emph{The kauffman polynomial of links and representation
  theory}, New Developments In The Theory Of Knots. Series: Advanced Series in
  Mathematical Physics, ISBN: 978-981-02-0162-3. WORLD SCIENTIFIC, Edited by
  Toshitake Kohno, vol. 11, pp. 480-493 \textbf{11} (1990), 480--493.

\bibitem[Ocn88]{Ocn88}
Adrian Ocneanu, \emph{Quantized groups, string algebras and galois theory for
  algebras}, Operator algebras and applications \textbf{2} (1988), 119--172.

\bibitem[Pop95]{Pop95}
Sorin Popa, \emph{An axiomatization of the lattice of higher relative
  commutants of a subfactor}, Inventiones mathematicae \textbf{120} (1995),
  no.~1, 427--445.

\bibitem[Ren17]{RenThesis}
Yunxiang Ren, \emph{Skein theory of planar algebras and some applications},
  Ph.D. thesis, Vanderbilt University, 2017.

\bibitem[Ren19]{Ren19Uni}
\bysame, \emph{Universal skein theory for group actions}, Advances in
  Mathematics \textbf{356} (2019), no.~7, 106840.

\end{thebibliography}
\bibliographystyle{amsalpha}

\end{document}